\documentclass{article}

\usepackage{amsmath,amssymb,amsfonts,amsthm}
\usepackage{color,graphicx,bbm,verbatim}
\usepackage{mathtools}
\usepackage{epstopdf}
\usepackage{algorithmic}
\usepackage{subfig}
\usepackage{booktabs}
\usepackage{dcolumn}
\usepackage[detect-all]{siunitx}
\usepackage{hyperref}
\usepackage[capitalize]{cleveref}
\usepackage{multirow}
\usepackage[a4paper,margin=1.2in]{geometry}

\usepackage{sectsty}

\numberwithin{equation}{section}

\sectionfont{\large}
\subsectionfont{\large}

\ifpdf
  \DeclareGraphicsExtensions{.eps,.pdf,.png,.jpg}
\else
  \DeclareGraphicsExtensions{.eps}
\fi

\title{Convergence rate analysis and improved iterations for numerical radius computation}

\author{Tim Mitchell\thanks{
Max Planck Institute for Dynamics of Complex Technical Systems, Sandtorstr. 1, 39106 Magdeburg, Germany \href{mailto:mitchell@mpi-magdeburg.mpg.de}{\texttt{mitchell@mpi-magdeburg.mpg.de}}.}
}
\date{{\small January 31st, 2020\\Revised: December 15th, 2020, October 27, 2021, June 6, 2022}}

\theoremstyle{plain}
\newtheorem{theorem}{Theorem}[section]
\newtheorem{lemma}[theorem]{Lemma}
\newtheorem{corollary}[theorem]{Corollary}

\newtheorem{assumption}[theorem]{Assumption}
\newtheorem{remark}[theorem]{Remark}
\newtheorem{keyremark}[theorem]{Key Remark}
\newtheorem{definition}[theorem]{Definition}
\newtheorem{example}[theorem]{Example}
\crefname{assumption}{Assumption}{Assumptions}

\usepackage{hyperref}
\usepackage{enumitem}
\usepackage{siunitx}
\usepackage{mathtools}
\usepackage{newfloat}
\usepackage[section]{algorithm}
\usepackage{algorithmic}
\usepackage{color}
\def\matlab{MATLAB}

\def\R{\mathbb{R}}
\def\C{\mathbb{C}}

\def\bigO{\mathcal{O}}
\def\imagunit{\mathbf{i}}

\DeclareMathOperator{\sgn}{sgn}

\DeclareMathOperator{\Arg}{Arg}
\DeclareMathOperator*{\argmax}{arg\,max}

\DeclareMathOperator{\arcsec}{arcsec}

\def\e{\mathrm{e}}
\def\eit{\eix{\theta}}
\def\emit{\emix{\theta}}

\newcommand{\eix}[1]{\e^{\imagunit #1}}
\newcommand{\emix}[1]{\e^{-\imagunit #1}}

\def\eps{\varepsilon}

\def\Hcal{\mathcal{H}}
\def\Hinf{\Hcal_\infty}

\renewcommand{\Re}{\mathrm{Re}\,}
\renewcommand{\Im}{\mathrm{Im}\,}

\def\ie{i.e., }
\def\eg{e.g., }

\def\beq{\begin{equation}}
\def\eeq{\end{equation}}
\def\beqs{\begin{equation*}}
\def\eeqs{\end{equation*}}
\def\bseq{\begin{subequations}}
\def\eseq{\end{subequations}}

\def\bd{\partial}
\def\fovsym{W}
\def\fovbdsym{\bd \fovsym}
\newcommand{\fovX}[1]{\fovsym(#1)}
\newcommand{\bfovX}[1]{\fovbdsym(#1)}
\def\fovA{\fovX{A}}
\def\bfovA{\bfovX{A}}

\def\numa{\alpha_\fovsym(A)}
\def\numr{r(A)}
\newcommand{\numaX}[1]{\alpha_\fovsym(#1)}

\DeclareFloatingEnvironment[placement=!ht]{algfloat}
\newcommand{\algnote}[1]{\footnotesize \sc{Note: \it#1 } }

\begin{document}
\maketitle

\begin{abstract}
The main 
two algorithms for computing the numerical radius are the level-set method of Mengi and Overton and
the cutting-plane method of Uhlig.
Via new analyses, we explain why
 the cutting-plane approach
is sometimes much faster or much slower than the level-set one  and then 
propose a new hybrid algorithm that remains efficient in all cases.
For matrices whose fields of values are a circular disk centered at the origin, 
we show that the 
cost 
of Uhlig's method  blows up with respect to the desired relative accuracy.
More generally, we also analyze the local behavior 
of Uhlig's cutting procedure at outermost points in the field of values, 
showing that it often has a fast Q-linear rate of convergence and is Q-superlinear at corners.
Finally, we identify and address inefficiencies in both the level-set and cutting-plane approaches
and propose refined versions of these~techniques.
\end{abstract}

\medskip

\noindent
\textbf{Notation:} $\| \cdot \|$ denotes the spectral norm, $\Lambda(\cdot)$ the spectrum (the set of eigenvalues) of a square matrix,
and $\lambda_\mathrm{max}(\cdot)$ and $\lambda_\mathrm{min}(\cdot)$, respectively, the largest and smallest eigenvalue of a Hermitian matrix.   $\e$ and $\imagunit$ respectively denote Euler's number and~$\sqrt{-1}$.

\section{Introduction}
Consider the discrete-time dynamical system
\beq
	\label{eq:ode_disc}
	x_{k+1} = Ax_k,
\eeq
where $A \in \C^{n \times n}$ and $x_k \in \C^n$.
The asymptotic behavior of \eqref{eq:ode_disc} is of course characterized by the moduli of $\Lambda(A)$.
Given the \emph{spectral radius} of~$A$,
\beq
	\rho(A) \coloneqq \max \{ |\lambda| : \lambda \in \Lambda(A)\},
\eeq
$\lim_{k \to \infty} \| x_k\| = 0$ for all $x_0$ if and only if $\rho(A) < 1$, with the asymptotic decay rate being faster
the closer $\rho(A)$ is to zero.  
However, knowing the transient behavior of~\eqref{eq:ode_disc} is often of interest.  
Clearly, the trajectory of~\eqref{eq:ode_disc} is tied to powers of~$A$, 
since $x_k  = A^k x_0$ and so \mbox{$\|x_k\| \leq \|A^k\| \|x_0\|$}.  Indeed,
a central theme of Trefethen and Embree's treatise on pseudospectra \cite{TreE05}
is how large $\sup_{k \geq 0} \|A^k\|$ can be.

One perspective is given by the \emph{field of values} (\emph{numerical range}) of $A$,
\beq
	\label{eq:fov}
	\fovA \coloneqq \{ x^* A x : x \in \C^n, \|x\| = 1\}.
\eeq
Consider the maximum of the moduli of points in $\fovA$, \ie the \emph{numerical radius}
\beq
	\label{eq:numr}
	\numr \coloneqq \max \{ | z | : z \in \fovA\}.
\eeq
It is known that $\tfrac{1}{2}\|A\| \leq r(A) \leq \|A\|$; see \cite[p.~44]{HorJ91}. 
Combining the lower bound with the power inequality $r(A^k) \leq (r(A))^k$ \cite{Ber65,Pea66} yields
\beq
	\label{eq:numr_ineq}
	\| A^k \| \leq 2 (\numr)^k.
\eeq
As $2(\numr)^k \leq \|A\|^k$ if and only if $\numr \leq \sqrt[k]{0.5} \|A\|$, and $\numr \leq \|A\|$ always holds, 
it follows that $2(\numr)^k$ is often a tighter upper bound for $\|A^k\|$ than $\|A\|^k$ is,
and so the numerical radius can be useful in estimating the transient behavior 
of \eqref{eq:ode_disc}.\footnote{Per \cite{TreE05},
the \emph{pseudospectral radius} and the \emph{Kreiss constant} \cite{Kre62}
also give information on the trajectory of~\eqref{eq:ode_disc}.
For computing these quantities, see \cite{MenO05,BenM19} and \cite{Mit20a,Mit21}.}

The concept of the numerical radius dates to at least 1961; see \cite[p.~1005]{LewO20}.
In~1978, Johnson noted that $\numr$ could be computed via his 
cutting-plane technique to approximate~$\fovA$,
but that a modified algorithm would likely be more efficient~\cite[Remark~3]{Joh78}.
Such geometric approaches estimate~$\fovA$ (or $\numr$) 
by computing a number of supporting hyperplanes
to sufficiently approximate the boundary of $\fovA$ (or regions of it);
supporting hyperplanes are computed using the much earlier 
Bendixson-Hirsch theorem~\cite{Ben02a} and fundamental results of Kippenhahn~\cite{Kip51}.
Results related to computing $\numr$ also appeared in the 1990s.  
Mathias showed that~$\numr$ can be obtained by solving a semidefinite program~\cite{Mat93},
but doing so is expensive.
Much faster algorithms were then proposed by He and Watson~\cite{Wat96,HeW97}, 
but these methods may not converge to $\numr$.

The 2000s saw further interest in computing $\numr$ with the following key results.
In 2005, Mengi and Overton gave a fast globally convergent method 
for~$\numr$~\cite{MenO05} by combining an idea of 
He and Watson \cite{HeW97}
with the level-set approach of Boyd, Balakrishnan, Bruinsma, and Steinbuch (BBBS) \cite{BoyB90,BruS90}
for computing the $\Hinf$ norm.
Although Mengi and Overton observed that their method converged quadratically,
this was only later proved in 2012 by G\"urb\"uzbalaban in his PhD thesis~\cite[section~3.4]{Gur12}.
Meanwhile, in 2009, Uhlig proposed a fast geometric approach 
to computing $\numr$~\cite{Uhl09} using cutting planes and a new greedy strategy.\footnote{In this 
same paper~\cite[section~3]{Uhl09}, Uhlig also 
discussed how Chebfun~\cite{DriHT14} can be used to reliably compute~$\numr$ with just a few lines of \matlab,
but that it is generally orders of magnitude slower than either his method 
or the one of Mengi and Overton; see also~\cite{GreO18}.}
A major benefit of cutting-plane methods is that they only require 
computing~$\lambda_\mathrm{max}$ of $n \times n$ Hermitian matrices.  
If $A$ is sparse, this can be done efficiently and reliably using, 
say, \texttt{eigs} in \matlab\@.
Hence, Uhlig's method can be used on large-scale problems while still being globally convergent. 
In contrast, at every iteration, the level-set approach requires 
solving a generalized eigenvalue problem of order~$2n$, which by standard convention on work 
complexity, is an atomic operation with~$\bigO(n^3)$~work.
While Uhlig noted that convergence of his method can sometimes be quite slow~\cite[p.~344]{Uhl09}, 
his experiments in the same paper showed several problems where 
his cutting-plane method was decisively faster Mengi and Overton's level-set method.

A key motivation for our work here is the class of problems where the numerical radius of parametric
matrices is optimized, such as feedback control; see \cite{LewO20} for more details and other applications.
Since minimizing the numerical radius is a nonsmooth optimization problem, 
optimization solvers generally will converge slowly and require many function evaluations.
During the course of optimization, 
the numerical radius will be computed for many different parameter choices, and the
shape and location of the field of values will be constantly changing.
Per~\cite{LewO20}, the solutions to such numerical radius optimization problems 
are often so-called \emph{disk matrices};
a disk matrix is one whose field of values is a disk centered at the origin.
As we elucidate in this paper, whether a cutting-plane method
is very fast or very slow is determined by the geometry of the field of values,
and for disk matrices, the overall cost of a cutting-plane method blows up with respect to the desired relative accuracy.
For optimizing the numerical radius, we thus would like to 
have a numerical radius method that remains efficient in all 
cases, which is what we propose here. 
Moreover, we want such consistent efficiency without sacrificing the precision afforded by the hardware.
As we show later, one can avoid the high costs of cutting-plane methods if one settles for only 
a few digits of accuracy, but doing so can adversely affect the quality and reliability of optimization.
Inaccuracy in the estimates of the numerical radius values can cause
optimization solvers to stagnate (e.g., line searches may break down),
while computed gradients, which are critical, of $\max$ functions 
can be totally inaccurate even when the function values are computed to, say, seven digits;
for more details, see \cite{BenM18a} and \cite{BenMO18}.

The paper is organized as follows. 
In~\cref{sec:background}, we give necessary preliminaries on the field of values, the numerical 
radius, and earlier $\numr$ algorithms. 
We then identify and address some inefficiencies in the level-set method
of Mengi and Overton and propose a faster variant in~\cref{sec:alg1}.
We analyze Uhlig's method in \cref{sec:rate}, deriving
(a) its overall cost when the field of values is a disk centered at the origin, and (b) 
a Q-linear local rate of convergence result for its cutting procedure.  
These analyses precisely show how, depending on the problem, 
Uhlig's method can be either extremely fast or extremely slow.
In \cref{sec:alg2}, we identify an inefficiency in Uhlig's cutting procedure 
and address it via a more efficient cutting scheme whose exact convergence rate
we also derive.
Putting all of this together, we present our new hybrid algorithm in \cref{sec:hybrid}.
We validate our results experimentally in \cref{sec:experiments} and give concluding remarks in~\cref{sec:conclusion}.

\section{Preliminaries}
\label{sec:background}
We will need the following well-known facts~\cite{Kip51,HorJ91}:

\begin{remark}
\label{rem:fov}
Given $A \in \C^{n \times n}$, 
\begin{enumerate}[leftmargin=30pt,label=(A\arabic*),font=\normalfont]
\item $\fovA \subset \C$ is a compact, convex set,
\item if $A$ is real, then $\fovA$ has real axis symmetry,
\item if $A$ is normal, then $\fovA$ is the convex hull of $\Lambda(A)$,
\item $\fovA = [\lambda_\mathrm{min}(A), \lambda_\mathrm{max}(A)]$ if and only if $A$ is Hermitian,
\item the boundary of $\fovA$, $\bfovA$, is a piecewise smooth algebraic curve,
\item if $v \in \bfovA$ is a point where $\bfovA$ is not differentiable, \ie a corner, then $v \in \Lambda(A)$.
	Corners always correspond to two line segments in $\bfovA$ meeting  at some angle less than $\pi$ radians.
\end{enumerate}
\end{remark}

\begin{definition}
Given a nonempty closed set $\mathcal{D} \subset \C$, 
a point $\tilde z \in \mathcal{D}$ is (globally) \emph{outermost} 
if $|\tilde z| = \max\{|z| : z \in \mathcal{D}\}$ and 
\emph{locally outermost} if $\tilde z$ is an outermost point of $\mathcal{D} \cap \mathcal{N}$, for some neighborhood $\mathcal{N}$ of $\tilde z$.
\end{definition}

For continuous-time systems $\dot x = Ax$, we have the \emph{numerical abscissa}
\beq
	\label{eq:numa}
	\numa \coloneqq \max \{ \Re z : z \in \fovA\},
\eeq
\ie the maximal real part of all points in $\fovA$.
Unlike the numerical radius, computing the numerical abscissa is straightforward, as~\cite[p.~34]{HorJ91}
\beq
	\label{eq:tan_line}
	\numa = \lambda_\mathrm{max} \left(\tfrac{1}{2} \left(A + A^*\right)\right).
\eeq
For $\theta \geq 0$, $\fovX{\eit A}$ is $\fovA$ rotated  counter-clockwise about the origin.  Consider
\beq		
	\label{eq:hmat}
	H(\theta) \coloneqq \tfrac{1}{2} \left(\eit A + \emit A^*\right),
\eeq
so $\numaX{\eit A} = \lambda_\mathrm{max}(H(\theta))$ and $\numa = \lambda_\mathrm{max}(H(0))$.
Let $\lambda_\theta$ and~$x_\theta$ denote, respectively, $\lambda_\mathrm{max}(H(\theta))$
and an associated normalized eigenvector.  Furthermore, let $L_\theta$ denote the line $\{ \emit (\lambda_\theta + \imagunit t): t \in \R\}$
and $P_\theta$ the half plane $\emit \{ z : \Re z \leq \lambda_\theta\}$.
Then $L_\theta$ is a \emph{supporting hyperplane} for~$\fovA$
and \cite[p.~597]{Joh78}
\begin{enumerate}[leftmargin=30pt,label=(B\arabic*),font=\normalfont]
\item $\fovA \subseteq P_\theta$ for all $\theta \in [0,2\pi)$,
\item $\fovA = \cap_{\theta \in [0,2\pi)} P_\theta$,
\item $z_\theta = x_\theta^*Ax_\theta \in L_\theta$ is a \emph{boundary point} of $\fovA$.
\end{enumerate}
As $H(\theta + \pi)  = -H(\theta)$,
$P_{\theta+\pi}$ can also be obtained via $\lambda_\mathrm{min}(H(\theta))$ and an associated eigenvector.
The Bendixson-Hirsch theorem is a special case of these properties, defining the bounding box
 of $\fovA$ for $\theta = 0$ and~\mbox{$\theta = \tfrac{\pi}{2}$}.
 
When $\lambda_\mathrm{max}(H(\theta))$ is simple, 
the following result of Fiedler~\cite[Theorem~3.3]{Fie81} 
gives a formula for computing the \emph{radius of curvature $\tilde r$ of $\bfovA$ at~$z_\theta$},
defined as the radius of the  \emph{osculating circle of $\bfovA$ at~$z_\theta$}, \ie
the circle with the same tangent and curvature as $\bfovA$ at $z_\theta$.
At corners of $\bfovA$, we say that~$\tilde r = 0$,
while at other boundary points where the radius of curvature is well defined,\footnote{An example
where $\tilde r$ is not well defined is given by $A = \begin{bsmallmatrix} J & 0 \\ 0 & J+I \end{bsmallmatrix}$
with $J = \begin{bsmallmatrix} 0 & 1 \\ 0 & 0 \end{bsmallmatrix}$.
At~\mbox{$b=0.5 \in \bfovA$}, 
two of the algebraic curves, a line segment
 and a semi-circle, comprising $\bfovA$
 meet, and $\bfovA$ is only once differentiable at this non-corner boundary point.
Here, the radius of curvature of $\bfovA$ jumps from $0.5$ (for the semi-circular piece) 
 to $\infty$ (for the line segment).} 
 $\tilde r > 0$ and becomes infinite at points inside line segments in $\bfovA$.
Although the formula is given for $\theta=0$ and~$z_\theta = 0$,
by simple rotation and shifting, it can be applied generally.
See \cref{fig:demo_fov} for a depiction of the osculating circle of $\bfovA$ at
an outermost point in~$\fovA$.  

\begin{theorem}[Fiedler]
\label{thm:curvature}
Let $H_1 = \tfrac{1}{2}(A + A^*)$, \mbox{$H_2 = \tfrac{1}{2\imagunit}(A - A^*)$},
let $H_1^+$ be the Moore-Penrose pseudoinverse of $H_1$,
 and let $x_\theta$ be a normalized eigenvector 
 corresponding to $\lambda_\mathrm{max}(H(\theta))$.
Noting that $A = H_1 + \imagunit H_2$ and $H(0) = H_1$,
suppose that \mbox{$\lambda_\mathrm{max}(H_1) = 0$} and is simple, 
and that the associated boundary point \mbox{$z_\theta = x_\theta^* A x_\theta = 0$}, where $\theta=0$.
Then the radius of curvature of $\bfovA$ at $z_\theta$ is
\beq
	\label{eq:curvature}
	\tilde r = -2 (H_2 x_\theta)^* H_1^+ (H_2 x_\theta).
\eeq
\end{theorem}

\begin{figure}
\centering
\subfloat[The field of values and local curvature.]{
\resizebox*{6.0cm}{!}{\includegraphics[trim=1.3cm 0cm 1.3cm 0cm,clip]{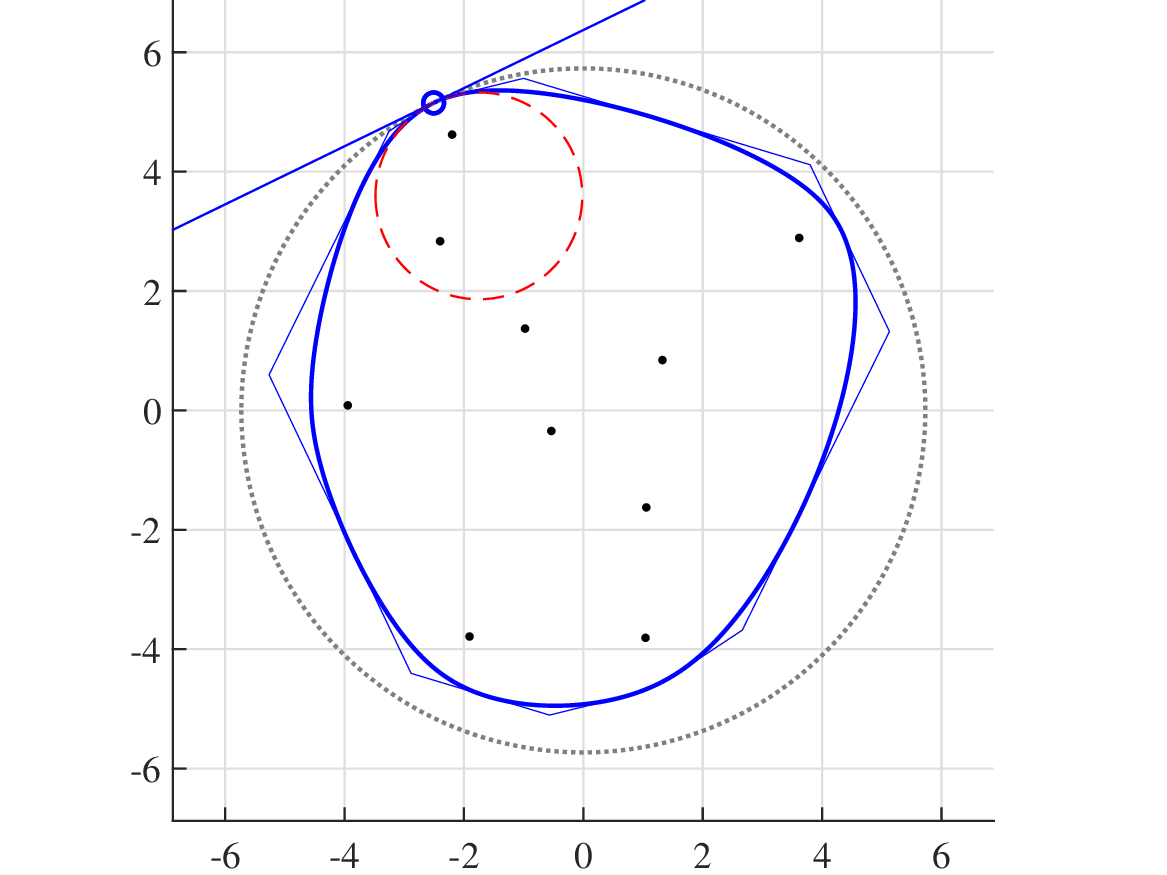}}
\label{fig:demo_fov}
}
\subfloat[Level-set iterates for $h(\theta)$.]{
\resizebox*{6.0cm}{!}{\includegraphics[trim=1.3cm 0cm 1.3cm 0cm,clip]{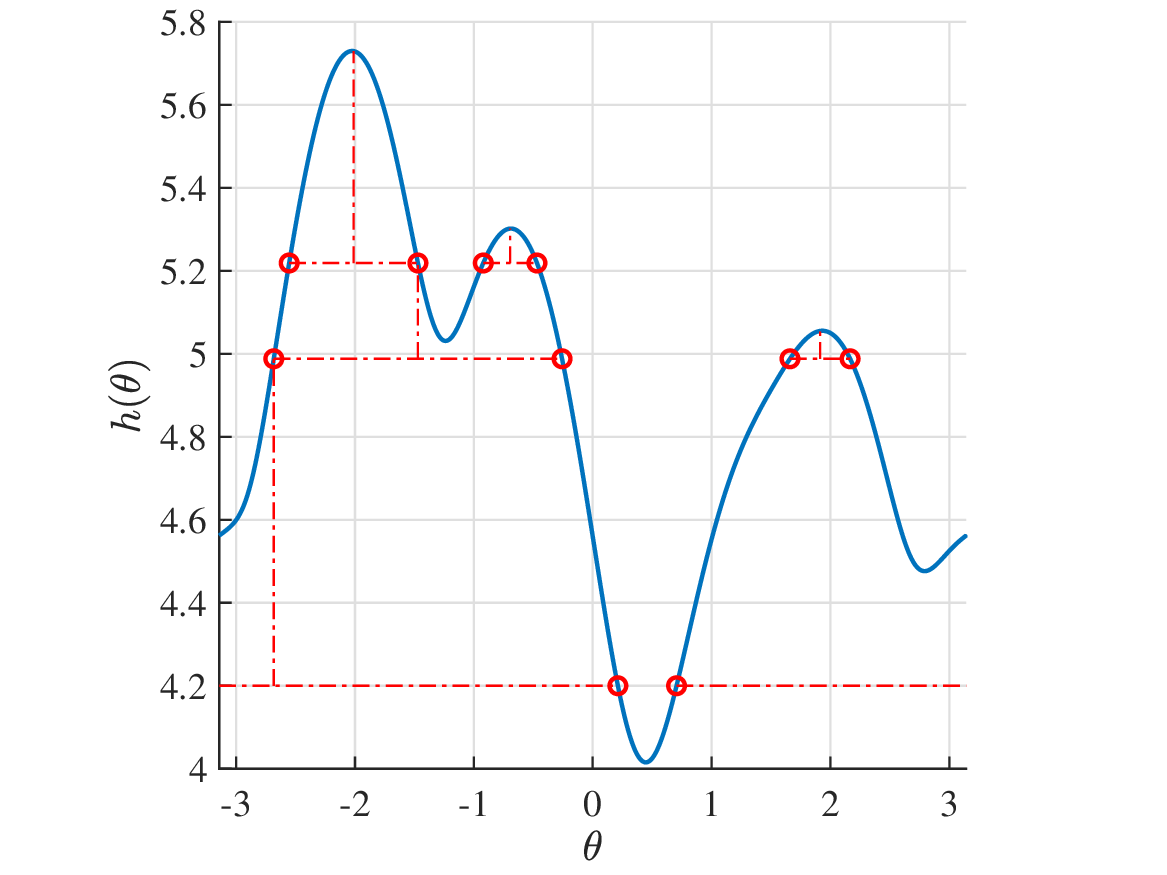}}
\label{fig:demo_levelset}
}
\caption{For a random matrix $A \in \C^{10 \times 10}$,
the left and right panes respectively show $\bfovA$ (blue curve) and $h(\theta)$ (blue plot). 
On the left, the following are also shown: $\Lambda(A)$ (black dots), 
a polygonal approximation $\mathcal{G}_j$ to $\fovA$ (blue polygon),
 the outermost point in $\fovA$ (small blue circle) with the corresponding 
 supporting hyperplane (blue line) and osculating circle (dashed red circle),
and the circle of radius~$\numr$ centered at the origin (black dotted circle).
On the right, three iterations of the level-set method are also shown (small circles and dash-dotted lines in red).
}
 \label{fig:demo}
\end{figure}

Via \eqref{eq:tan_line} and \eqref{eq:hmat}, 
the numerical radius can be written as
\beq
	\label{eq:nr_opt}
	\numr 
	= \max_{\theta \in [0,2\pi)} h(\theta) 
	\qquad \text{where} 
	\qquad
	h(\theta) \coloneqq \lambda_\mathrm{max} \left(H(\theta)\right),
\eeq
\ie a one-variable maximization problem.
Via $H(\theta + \pi)  = -H(\theta)$, it also follows that 
\beq	
	\label{eq:nr_opt_abs}
	\numr = \max_{\theta \in [0,\pi)} \rho( H(\theta)).
\eeq
 However, as \eqref{eq:nr_opt} and \eqref{eq:nr_opt_abs} may have multiple maxima, 
 it is not straightforward to find
a global maximizer of either, and crucially, assert that it is indeed a global maximizer in order to verify that $\numr$ has been computed.
Per (A3), we generally assume that $A$ is non-normal, as otherwise $\numr = \rho(A)$.

We now discuss earlier numerical radius algorithms in more detail.
In 1996, Watson proposed two $\numr$ methods~\cite{Wat96}:
one which converges to local maximizers of~\eqref{eq:nr_opt_abs}
and a second which lacks convergence guarantees but is cheaper 
(though they each do $\bigO(n^2)$ work per iteration).  
However, as both iterations are related to the power
method, they may exhibit very slow convergence, and the cheaper iteration may
not converge at all.  Shortly thereafter~\cite{HeW97}, He and Watson used 
the second iteration (because it was cheaper) in combination with a new certificate test inspired by Byers’ distance to 
instability algorithm~\cite{Bye88}.  This certificate either asserts that $\numr$ has been
computed to a desired accuracy or provides a way to restart Watson’s cheaper 
iteration with the hope of more accurately estimating~$\numr$.  
However, He and Watson’s method is still not guaranteed to converge, 
since Watson’s cheaper iteration may not converge.  
Inspired by the BBBS algorithm for computing the $\Hinf$~norm,
Mengi and Overton then proposed a globally convergent iteration for~$\numr$ in 2005
by using He and Watson's certificate test in a much more powerful way.
Given $\gamma \leq \numr$, the test actually allows one to obtain the $\gamma$-level set 
of $h(\theta)$, \ie $\{ \theta : h(\theta) = \gamma\}$.   
Assuming the level set is not empty, Mengi and Overton's method then evaluates~$h(\theta)$ 
at the midpoints of the intervals under~$h(\theta)$ determined by the $\gamma$-level~set points.
Estimate $\gamma$ is then updated (increased) to the highest of these corresponding function values.
This process is done in a loop, and as mentioned in the introduction, has local quadratic convergence.  
See \cref{fig:demo_levelset} for a depiction of this level-set iteration.

The certificate (or level-set) test is based on \cite[Theorem~3.1]{MenO05}, 
which is a slight restatement of \cite[Theorem~2]{HeW97} from He and Watson.
We omit the proof. 
\begin{theorem}
\label{thm:level}
Given $\gamma \in \R$, 
the pencil $R_\gamma - \lambda S$ has $\eit$ as an eigenvalue or is singular if and only if
$\gamma$ is an eigenvalue of $H(\theta)$ defined in \eqref{eq:hmat}, where
\beq
	\label{eq:RS}
	R_\gamma \coloneqq
	\begin{bmatrix}
		2\gamma I  	& -A^* \\
		I 			& 0  
	\end{bmatrix}
	\quad \text{and} \quad
	S \coloneqq
	\begin{bmatrix}
		A   		& 0 \\
		0 		& I  
	\end{bmatrix}.
\eeq
\end{theorem}
Per \cref{thm:level}, the $\gamma$-level set of $h(\theta) = \lambda_\mathrm{max}(H(\theta))$ 
is associated with the unimodular eigenvalues of $R_\gamma - \lambda S$, 
which can be obtained in $\bigO(n^3)$ work (with a significant constant factor).
Note that the converse may not hold, \ie
for a unimodular eigenvalue of $R_\gamma - \lambda S$,
$\gamma$ may correspond to an eigenvalue of $H(\theta)$ other than $\lambda_\mathrm{max}(H(\theta))$.
Given any $\theta \in [0,2\pi)$, also note that $R_\gamma - \lambda S$ is nonsingular for all $\gamma > h(\theta)$.
This is because if $\fovA$ and a disk centered at the origin enclosing $\fovA$ have more than $n$ shared boundary points,
then $\fovA$ is that disk; see \cite[Lemma~6]{TamY99}.

Uhlig's method computes $\numr$ via updating a bounded convex polygonal approximation 
to $\fovA$ and set of known points in $\bfovA$ respectively given by:
\beqs
	\mathcal{G}_j 
	\coloneqq
	\bigcap_{\theta \in \left\{\theta_1,\ldots,\theta_j\right\}} P_\theta
	\qquad \text{and} \qquad
	\mathcal{Z}_j \coloneqq
	\{ z_{\theta_1}, \ldots, z_{\theta_j} \},
\eeqs
where $\fovA \subseteq \mathcal{G}_j$ (see~\cref{fig:demo_fov} for a depiction), 
$0 \leq \theta_1 < \cdots < \theta_j < 2\pi$,
and 
$z_{\theta_\ell} = x_{\theta_\ell}^* A x_{\theta_\ell}$ 
is a boundary point of $\fovA$ on $L_{\theta_\ell}$
for $\ell = 1,\ldots,j$.
Note that the corners of $\mathcal{G}_j$ are given by 
$L_{\theta_{\ell}} \cap L_{\theta_{\ell+1}}$ for $\ell = 1,\ldots,j-1$ and $L_{\theta_1} \cap L_{\theta_j}$.
Given $\mathcal{G}_j$ and $\mathcal{Z}_j$, lower and upper bounds 
$l_j \leq \numr \leq u_j$
are immediate, where
\[
	l_j \coloneqq  \max \{ |b| : b \in \mathcal{Z}_j \}
	\quad \text{and} \quad
	u_j \coloneqq  \max \{ |c| : c \text{ a corner of } \mathcal{G}_j \},
\]
so we define the relative error estimate:
\beq
	\label{eq:nr_err}
	\eps_j \coloneqq 
	\frac{ u_j - l_j }{ l_j}.
\eeq
By repeatedly cutting outermost corners of $\mathcal{G}_j$, and in turn, 
adding computed boundary points of $\fovA$ to $\mathcal{Z}_j$, 
it follows that $\eps_j$ must fall 
below a desired relative tolerance for some~$k \geq j$; 
hence, $\numr$ can be computed to any desired accuracy.
Uhlig's method achieves this via a greedy strategy. 
On each iteration, his algorithm chops off an outermost corner $c_j$ from $\mathcal{G}_j$,
which is done via computing the 
supporting hyperplane~$L_{\theta_{j+1}}$ for~\mbox{$\theta_{j+1} = -\Arg(c_{j})$}
and the boundary point $z_{\theta_{j+1}} = x_{\theta_{j+1}}^* A x_{\theta_{j+1}}$.
Assuming that $c_j \not\in \fovA$, the cutting operation results in 
\mbox{$\mathcal{G}_{j+1} \coloneqq \mathcal{G}_j \cap P_{\theta_{j+1}}$}, 
a smaller polygonal region excluding the corner~$c_j$,
and 
\mbox{$\mathcal{Z}_{j+1} \coloneqq \mathcal{Z}_j \cup \{z_{\theta_{j+1}}\}$};
therefore, $\eps_{j+1} \leq \eps_j$.
However, if~$c_j$ happens to be a corner of $\bfovA$, then it cannot be cut from $\mathcal{G}_j$, and 
instead this operation asserts that $|c_j| = \numr$, 
and so $\numr$ has been computed. 
In \cref{sec:rate}, \Cref{fig:uhlig} depicts Uhlig's method when a corner is cut.

\begin{remark}
\label{rem:two_hp}
Recall that the parallel supporting hyperplane $L_{\theta_{j+1} + \pi}$ 
and the corresponding boundary point 
$z_{\theta_{j+1} + \pi}$
can be obtained via an eigenvector $\tilde x_{\theta_{j+1}}$ of $\lambda_\mathrm{min}(H(\theta_{j+1}))$.
If~$\tilde x_{\theta_{j+1}}$ is already available or relatively cheap to compute,
there is little reason \emph{not} to also update~$\mathcal{G}_j$ 
and~$\mathcal{Z}_j$ using this additional information.
\end{remark}

\section{Improvements to the level-set approach}
\label{sec:alg1}
We now propose two straightforward but important modifications  
to make the level-set approach faster and more reliable.
We need the following immediate corollary of \cref{thm:level},
which clarifies that \cref{thm:level} also allows all points in any $\gamma$-level set 
of $\rho(H(\theta))$ to be computed.
\begin{corollary}
\label{cor:remap}
Given $\gamma \geq 0$, if $\rho(H(\theta)) = \gamma$, then there exists 
$\lambda \in \C$ such that $|\lambda| = 1$, $\det(R_\gamma - \lambda S) = 0$,
and $\theta = f(\Arg(\lambda))$,
where $f : (-\pi,\pi] \mapsto [0, \pi)$ is 
\vspace{-0.25cm}
\beq
	\label{eq:remap}
	f(\theta) \coloneqq
	\begin{cases}
	\theta + \pi	& \text{if } \theta < 0 \\
	0 			& \text{if } \theta = 0 \text{ or } \theta = \pi \\
	\theta 		& \text{otherwise}.
	\end{cases}
\eeq
\end{corollary}

Thus, first we propose doing a BBBS-like iteration 
using $\rho(H(\theta))$ instead of $h(\theta)$,
which also has local quadratic convergence.
By an extension of the argument of Boyd and Balakrishnan~\cite{BoyB90}, near maximizers, 
$\rho(H(\theta))$ is unconditionally twice continuously differentiable with Lipschitz second derivative; 
see~\cite{MitO21}.
Using $\rho(H(\theta))$ is also typically faster in terms of constant factors.
This is because $\rho(H(\theta)) \geq h(\theta)$ always holds, $\rho(H(\theta)) \geq 0$ (unlike $h(\theta)$, which can be negative),
and the optimization domain is reduced from $[0,2\pi)$ to $[0,\pi)$.
Thus, every update to the current estimate~$\gamma$ computed via 
$\rho(H(\theta))$ must be at least as good as the one from using~$h(\theta)$ (and possibly much better), 
and there may also be fewer level-set intervals per iteration, 
which reduces the number of eigenproblems incurred involving $H(\theta)$.

Second, we also propose using local optimization on top of the BBBS-like step at every iteration, \ie 
the BBBS-like step is used to initialize optimization in order to find a maximizer of $\rho(H(\theta))$.
The first benefit  is speed, as optimization often results in much larger 
updates to estimate $\gamma$ and these updates are now locally optimal.
This greatly reduces the total number of expensive eigenvalue computations done with 
$R_\gamma - \lambda S$, often down to just one;
hence, the overall runtime can be substantially reduced since in comparison,
optimization is cheap (as we explain momentarily).
The second benefit is that using optimization also avoids some numerical difficulties 
when solely working with $R_\gamma - \lambda S$ to update $\gamma$.
In their 1997 paper, He and Watson showed  
that the condition number of a unimodular eigenvalue of $R_\gamma - \lambda S$ actually blows up
as $\theta$ approaches critical values of $h(\theta)$ or $\rho(H(\theta))$ \cite[Theorem~4]{HeW97},\footnote{The 
exact statement appears in the last lines of the corresponding proof on p.~335.}
as this corresponds to a pair of unimodular eigenvalues of $R_\gamma - \lambda S$ coalescing
into a double eigenvalue.
Since this must always occur as a level-set method converges, 
rounding errors may prevent all of the unimodular eigenvalues from being detected,
causing level-set points to go undetected, thus resulting in stagnation of
the algorithm before it finds~$\numr$ to the desired accuracy.
He and Watson wrote that their analytical result was ``hardly encouraging" \cite[p.~336]{HeW97},
though they did not observe this issue in their experiments.
However, an example of such a deleterious effect is shown in~\cite[Figure~2]{BenM19},
where analogous eigenvalue computations are shown to greatly reduce numerical accuracy
when computing the \emph{pseudospectral abscissa}~\cite{BurLO03}.

In contrast, optimizing $\rho(H(\theta))$ does not lead to numerical difficulties.
This objective function is both Lipschitz (as $H(\theta)$ is Hermitian \cite[Theorem~II.6.8]{Kat82})
and smooth at its maximizers (as discussed above).
Thus, local maximizers of $\rho(H(\theta))$
can be found using, say, Newton's method, with only a handful of iterations.
Interestingly, in their concluding remarks \cite[p.~341--2]{HeW97}, He and Watson 
seem to have been somewhat pessimistic about using Newton's method, writing that while it would have
faster local convergence than Watson's iteration, 
``the price to be paid is at least a considerable increase in computation,
and possibly the need of the calculation of higher derivatives, and for the incorporation
of a line search."  
As we now explain, using, say, secant or Newton's method, is 
actually an overall big win.  Also, note that with either secant or Newton, 
steps of length one are always eventually accepted; hence, the cost of line searches 
should not be a concern. 

\begin{algfloat}[t]
\begin{algorithm}[H]
\floatname{algorithm}{Algorithm}
\caption{An Improved Level-Set Algorithm}
\label{alg1}
\begin{algorithmic}[1]
	\REQUIRE{  
		$A \in \C^{n \times n}$ with $n \geq 2$, 
		initial guesses $\mathcal{M} = \{\theta_1,\ldots,\theta_q\}$, and \mbox{$\tau_\mathrm{tol} > 0$}.
	}
	\ENSURE{ 
		$\gamma$ such that $|\gamma -\numr| \leq \tau_\mathrm{tol} \cdot \numr$.
		\\ \quad
	}
	
	\STATE $\mathcal{M} \gets \mathcal{M} \cup \{ 0 \}$ 
	\WHILE { $\mathcal{M}$ is not empty } 
		\STATE $\theta_\mathrm{BBBS} \gets \argmax_{\theta \in \mathcal{M}} \rho(H(\theta))$ \COMMENT{In case of ties, just take any one}
		\STATE $\gamma \gets$ maximization of $\rho(H(\theta))$ via local optimization initialized at $\theta_\mathrm{BBBS}$
		\STATE $\gamma \gets \gamma (1 + \tau_\mathrm{tol})$
		\STATE $\Theta \gets
				\{ f(\Arg(\lambda)) : \det (R_\gamma - \lambda S) = 0, |\lambda | = 1 \}$
		\STATE $[\theta_1,\ldots,\theta_q] \gets 
			\Theta~\text{sorted in increasing order with any duplicates removed}$
		\STATE $\mathcal{M} \gets \{ \theta : \rho(H(\theta)) > \gamma ~\text{where}~ \theta = 0.5(\theta_\ell + \theta_{\ell+1}), \ell=1,\ldots,q-1\}$
	\ENDWHILE
\end{algorithmic}
\end{algorithm}
\vspace{-0.4cm}
\algnote{
For simplicity, we forgo giving pseudocode to exploit possible normality of~$A$ or symmetry of~$\fovA$,
and assume that eigenvalues and local maximizers are obtained exactly and that the optimization solver is monotonic,
\ie it guarantees $\rho(H(\theta_\mathrm{BBBS})) \leq \rho(H(\theta_\star))$, where $\theta_\star$ is the maximizer 
computed in line~4.
Recall that~$f(\cdot)$ is defined in~\eqref{eq:remap}, and note that the method reduces to a BBBS-like iteration using \eqref{eq:nr_opt_abs}
if line~4 is replaced by \mbox{$\gamma \gets \rho(H(\theta_\mathrm{BBBS}))$}.
Running optimization from other angles in $\mathcal{M}$ (in addition to $\theta_\mathrm{BBBS}$) every iteration 
may also be advantageous, particularly if this can be done via parallel processing.
Adding zero to the initial set $\mathcal{M}$ avoids having to deal 
with any ``wrap-around" level-set intervals due to the periodicity of $\rho(H(\theta))$.
}
\end{algfloat}

Suppose $\rho(H(\theta))$ is attained by a unique eigenvalue~$\lambda_j$ 
with normalized eigenvector $x_j$.
Then by standard perturbation theory for simple eigenvalues, 
\beq
	\rho^\prime(H(\theta)) = \sgn(\lambda_j) \cdot x_j^*H^\prime(\theta)x_j = \sgn(\lambda_j) \cdot x_j^*\left(\tfrac{\imagunit}{2} \left( \eit A - \emit A^* \right) \right) x_j.
\eeq
Thus, given $\lambda_j$ and $x_j$, the additional cost of obtaining $\rho^\prime(H(\theta))$ 
mostly amounts to the single matrix-vector product $H^\prime(\theta)x_j$.
To compute $\rho^{\prime\prime}(H(\theta))$, we will need the following  result for second derivatives
of eigenvalues; see \cite{Lan64}.
\begin{theorem}
\label{thm:eigdx2}
For $t \in \R$, let $A(t)$ be a twice-differentiable $n \times n$ Hermitian matrix family
with, for $t=0$, eigenvalues $\lambda_1 \geq \ldots \geq \lambda_n$ and 
associated eigenvectors $x_1,\ldots,x_n$, with $\|x_k\| = 1$ for all~$k$.
Then assuming $\lambda_j$ is unique,
\[
	\lambda_j^{\prime\prime}(t) \bigg|_{t=0}= x_j^* A''(0) x_j + 2 \sum_{k \ne j} \frac{| x_k^* A'(0) x_j |^2}{\lambda_k - \lambda_j}.
\]
\end{theorem}

Although obtaining the eigendecomposition of $H(\theta)$ is cubic work, 
this is generally negligible compared to the cost of obtaining all the unimodular eigenvalues of $R_\gamma - \lambda S$
when using \cref{thm:level} computationally;
recall that $H(\theta)$ is an $n \times n$ Hermitian matrix, 
while $R_\gamma - \lambda S$ is a generalized eigenvalue problem of order $2n$.
Moreover, $H^\prime(\theta)x_j$ would already be computed for~$\rho^\prime(H(\theta))$,
while
$ 	
	H^{\prime\prime}(\theta)x_j = -H(\theta)x_j = -\lambda_j x_j
$,
so there is no other work of consequence to obtain $\rho^{\prime\prime}(H(\theta))$ via \cref{thm:eigdx2}.

\begin{table}
\centering
\footnotesize 
\caption{The running time of a given operation \emph{divided by} the running time of \texttt{eig($H(\theta)$)}
for random $A \in \C^{n \times n}$.
Eigenvectors were requested for~$H(\theta)$ (for computing derivatives and boundary points)
but not for $R_\gamma - \lambda S$.
For \texttt{eigs}, \texttt{k} is the number of eigenvalues requested,
while \texttt{'LM'} (largest modulus), \texttt{'LR'} (largest real), and \texttt{'BE'} (both ends)
specifies which eigenvalues are desired.
}
\setlength{\tabcolsep}{3pt} 
\begin{tabular}{c | r SS | SS | SSS | S} 
\toprule
\multicolumn{2}{c}{} &
\multicolumn{2}{c}{\texttt{eigs($H(\theta)$,k,'LM')}} &
\multicolumn{2}{c}{\texttt{eigs($H(\theta)$,k,'LR')}} &
\multicolumn{3}{c}{\texttt{eigs($H(\theta)$,k,'BE')}} & 
\multicolumn{1}{c}{\texttt{eig($R_\gamma$,$S$)}}  \\
\cmidrule(lr){3-4}
\cmidrule(lr){5-6}
\cmidrule(lr){7-9}
\cmidrule(lr){10-10}
\multicolumn{1}{c}{} &
\multicolumn{1}{c}{$n$} & 
\multicolumn{1}{c}{$\texttt{k}=1$} & 
\multicolumn{1}{c}{$\texttt{k}=6$} & 
\multicolumn{1}{c}{$\texttt{k}=1$} & 
\multicolumn{1}{c}{$\texttt{k}=6$} & 
\multicolumn{1}{c}{$\texttt{k}=2$} & 
\multicolumn{1}{c}{$\texttt{k}=4$} & 
\multicolumn{1}{c}{$\texttt{k}=6$} & 
\multicolumn{1}{c}{} \\ 
\midrule
\multirow{4}{*}{\rotatebox[origin=c]{90}{Dense $A$}}
& \multicolumn{1}{r|}{200}  &   6.2 &   1.9 &   1.0 &   1.2 &   1.9 &   1.2 &   1.2 &  33.9 \\ 
& \multicolumn{1}{r|}{400}  &   0.6 &   0.9 &   0.5 &   0.8 &   0.6 &   0.7 &   0.8 &  75.2 \\ 
& \multicolumn{1}{r|}{800}  &   0.7 &   1.3 &   0.6 &   1.1 &   1.3 &   1.1 &   1.1 & 175.0 \\ 
& \multicolumn{1}{r|}{1600} &   0.4 &   0.6 &   0.3 &   0.6 &   0.6 &   0.7 &   0.6 & 205.1 \\ 
\midrule
\multirow{4}{*}{\rotatebox[origin=c]{90}{Sparse $A$}}
& \multicolumn{1}{r|}{200}  &   4.6 &   3.7 &   1.9 &   3.3 &   3.3 &   3.6 &   3.6 &  42.8 \\ 
& \multicolumn{1}{r|}{400}  &   2.0 &   3.1 &   1.7 &   3.0 &   3.1 &   3.2 &   3.1 &  84.1 \\ 
& \multicolumn{1}{r|}{800}  &   0.8 &   1.2 &   0.7 &   1.2 &   1.4 &   1.1 &   1.3 & 172.7 \\ 
& \multicolumn{1}{r|}{1600} &   0.4 &   0.7 &   0.4 &   0.7 &   0.8 &   0.9 &   0.7 & 199.0 \\ 
\bottomrule
\end{tabular} 
\label{tbl:eig}
\end{table}

Pseudocode
for our improved level-set algorithm is given in~\cref{alg1}.
We now address some implementation concerns.
What method is used to find maximizers of~$\rho(H(\theta))$
depends on the relative costs of solving eigenvalue problems 
involving $H(\theta)$ and $R_\gamma - \lambda S$.
\cref{tbl:eig} shows examples where 34--205 calls of 
\texttt{eig($H(\theta)$)} can be done before the total cost exceeds that of 
a single call of \texttt{eig($R_\gamma$,$S$)}.
This highlights just how beneficial it can be to incur a few more computations with $H(\theta)$ to find local maximizers
as \cref{alg1} does.
Comparisons for computing extremal eigenvalues of~$H(\theta)$ via \texttt{eig} and \texttt{eigs}
are also shown in the table.  Such data inform whether
or not the increased cost of needing to use \texttt{eig} in order to compute
$\rho^{\prime\prime}(H(\theta))$ is offset by the advantages that second derivatives can bring,
\eg faster local convergence.  
Of course, fine-grained implementation decisions like these should ideally be made via tuning, 
as such timings are generally also software and hardware dependent.
Nevertheless, \cref{tbl:eig} suggests that implementing \cref{alg1} using Newton's method via \texttt{eig}
might be a bit more efficient than using the secant method for $n \leq 800$ or so.\footnote{
Subspace methods such as \cite{KreLV18} might also be used to find local maximizers of $h(\theta)$ or~$\rho(H(\theta))$
and would likely provide similar benefits in terms of accelerating 
the globally convergent algorithms in this paper.
}

There is one more subtle but important detail for implementing \cref{alg1}.
Suppose that $\theta_\mathrm{BBBS}$ in line~3 is close to the argument of a (nearly) double unimodular eigenvalue of $R_\gamma - \lambda S$, where $\gamma = \rho(H(\theta_\mathrm{BBBS}))$.
If rounding errors prevent this one or two eigenvalues from being detected as unimodular,
the computed $\gamma$-level set of $\rho(H(\theta))$ may be incomplete,
which again, can cause stagnation.  
As pointed out in~\cite[p.~372--373]{BurLO03} in the context of computing the pseudospectral abscissa,
a robust fix is simple: explicitly add $a(\theta_\mathrm{BBBS})$ to $\Theta$ in line~6
if it appears to be missing.

\section{Analysis of Uhlig's method}
\label{sec:rate}
In the next two subsections, we respectively 
(a) analyze the overall cost of Uhlig's method for so-called
\emph{disk matrices} and 
(b) for general problems, 
establish how the exact Q-linear local rate of convergence of Uhlig's cutting strategy
varies with respect to the local curvature of $\bfovA$ at outermost points.
A disk matrix is one whose field of values is a circular disk centered at the origin,
and it is a worst-case scenario for Uhlig's method; as we show in this case,
the required number of supporting hyperplanes to compute $\numr$ blows up 
with respect to increasing the desired relative accuracy.
Although relatively rare, disk matrices naturally arise from minimizing the numerical radius
of parametrized matrices; see~\cite{LewO20} for a thorough discussion.
For concreteness here, we make use of the $n \times n$ Crabb matrix:
\beq
	\label{eq:crabb}
	K_2 = \begin{bmatrix}
	0 & 2 \\
	0 & 0 
	\end{bmatrix},
	\ 
	K_3 = \begin{bmatrix}
	0 & \sqrt{2} & 0  \\
	0 & 0 & \sqrt{2} \\ 
	0 & 0 & 0
	\end{bmatrix},
	\ 
	K_n = 
	\begin{bmatrix}
	0 & \sqrt{2} &  & & & \\
	& . & 1 & & &  \\
	& & . & . & &  \\
	& & & . & 1 & \\
	& & &  & . & \sqrt{2} \\
	& & & & & 0 \\
	\end{bmatrix},
\eeq
where for all $n$, $r(K_n) = 1$ and $W(K_n)$ is the unit disk.
However, note that not all disk matrices are variations of Jordan blocks corresponding 
to the eigenvalue zero.  
For other types of disk matrices and the history and relevance of $K_n$, see~\cite{LewO20}.

\subsection{Uhlig's method for disk matrices}
The following theorem completely characterizes the total cost of Uhlig's method
for disk matrices with respect to a desired relative tolerance.
Note that Uhlig's method begins with a rectangular approximation~$\mathcal{G}_4$ to $\fovA$,
which for a disk matrix, is a square centered at the origin.

\begin{theorem}
\label{thm:uhlig_disk}
Suppose that $A \in \C^{n \times n}$ is a disk matrix with $\numr > 0$
and that~$\fovA$ is approximated by $\mathcal{G}_j$ with $j \geq 3$
and $\mathcal{G}_j$ a regular polygon, \ie 
it is the intersection of $j$ half planes $P_{\theta_\ell}$, where
$\theta_\ell = \tfrac{2\pi }{j} \ell$ for $\ell = 1,\ldots,j$.  
Then,
\begin{enumerate}[label=(\roman*),font=\normalfont]
\item $\eps_j = \sec(\pi/j) - 1$, 
\item if $\eps_j \leq \tau_\mathrm{tol}$,
	then $j \geq  \left\lceil \tfrac{\pi}{\arcsec(1 + \tau_\mathrm{tol})} \right\rceil$,
	where $\tau_\mathrm{tol} > 0$ is the desired relative error.
\end{enumerate} 
Moreover, if $\mathcal{G}_k$ is a further refined version of $\mathcal{G}_j$, 
so $\fovA \subseteq \mathcal{G}_k \subseteq \mathcal{G}_j$, then
\begin{enumerate}[label=(\roman*),font=\normalfont,resume]
\item if $\eps_k < \eps_j$, then $k \geq 2j$.
\item if $\eps_k \leq \tau_\mathrm{tol} < \eps_j$, then $k \geq j \cdot 2^d$,
	where 
	$d = \left\lceil \log_2 \left(\tfrac{\pi}{j \arcsec(1 + \tau_\mathrm{tol})}\right) \right\rceil$.
\end{enumerate}
\end{theorem}
\begin{proof}
As $\fovA$ is a disk centered at zero with radius $\numr$ and 
$\bfovA$ is a circle inscribed in the regular polygon $\mathcal{G}_j$,
every boundary point in $\mathcal{Z}_j$ has modulus~$\numr$, and so $l_j = \numr$,
and the moduli of the corners of~$\mathcal{G}_j$ are all identical.
Consider the corner $c$ with $\Arg(c) = \pi/j$
and the right triangle defined by zero, $\numr$ on the real axis, and $c$.
Then~\mbox{$|c| = u_j = \numr  \sec(\pi/j)$}, and so~(i) holds.
Statement~(ii) simply holds by substituting~(i) into $\eps_j \leq \tau_\mathrm{tol}$
and then solving for $j$.
For~(iii), as~$\eps_j = |c|$ for any corner~$c$ of~$\mathcal{G}_j$, 
all $j$~corners of $\mathcal{G}_j$ must be refined to lower the error;
thus,~$\mathcal{G}_k$ must have at least~$2j$~corners.
Finally, as $\lim_{j \to \infty} \eps_j = 0$,
but the error only decreases when~$j$ is doubled, 
it follows that in order for $\eps_k \leq \tau_\mathrm{tol}$ to hold,
$k \geq j \cdot 2^d$ for some~$d \geq 1$.
The smallest possible integer is obtained by replacing~$j$ in~(ii) with $j \cdot 2^d$ 
and solving for~$d$, thus proving~(iv).
\end{proof}

\begin{table}
\centering
\footnotesize
\caption{
For any disk matrix $A$, 
the minimum number of supporting hyperplanes required to compute $\numr$ to different accuracies is shown, 
where $\texttt{eps} \approx 2.22~\times~10^{-16}$.
}
\setlength{\tabcolsep}{6pt} 
\begin{tabular}{l | *{10}{S[table-format=9.0]}} 
\toprule
\multicolumn{1}{c}{} &
\multicolumn{3}{c}{\# of supporting hyperplanes needed}\\
\cmidrule(lr){2-4}
\multicolumn{1}{c}{}& 
\multicolumn{1}{c}{Minimum} & 
\multicolumn{2}{c}{Uhlig's method}\\ 
\cmidrule(lr){2-2}
\cmidrule(lr){3-4}
\multicolumn{1}{c}{Relative Tolerance} & 
\multicolumn{1}{c}{}& 
\multicolumn{1}{c}{Starting with $\mathcal{G}_3$} &  
\multicolumn{1}{c}{Starting with $\mathcal{G}_4$}  \\
\midrule
$\tau_\mathrm{tol} = \texttt{1e-4}$   &        223 &        384 &        256 \\ 
$\tau_\mathrm{tol} = \texttt{1e-8}$   &      22215 &      24576 &      32768 \\ 
$\tau_\mathrm{tol} = \texttt{1e-12}$  &    2221343 &    3145728 &    4194304 \\ 
$\tau_\mathrm{tol} = \texttt{eps}$    &  149078414 &  201326592 &  268435456 \\ 
\bottomrule
\end{tabular} 
\label{tbl:disk_thm}
\end{table}

Via \cref{thm:uhlig_disk}, we report the number of 
supporting hyperplanes needed to compute the numerical radius of disk matrices
for increasing levels of accuracy in~\cref{tbl:disk_thm}, 
illustrating just how quickly the cost of Uhlig's method skyrockets.
Combined with the timing data from~\cref{tbl:eig},
it is clear that the level-set approach will typically be much faster for disk matrices
or those whose fields of values are close to a disk centered at zero; 
indeed, since $h(\theta)=\rho(H(\theta))$ is constant for disk matrices, it converges in a single iteration.

\subsection{Local rate of convergence of Uhlig's cuts}
\label{sec:uhlig_conv}
As we now explain, the local behavior of Uhlig's cutting procedure
at an outermost point in~$\fovA$
can actually be understood by analyzing one key example.  
Note that we are making a distinction here between Uhlig's method as a whole
and his cutting procedure, since we need the notion of the latter 
for our local analysis; we also use  \emph{cutting strategy} or simply just \emph{cuts}
as synonyms for \emph{cutting procedure}.
For this analysis, we use Q-linear and Q-superlinear convergence,
where ``Q" stands for ``quotient"; see~\cite[p.~619]{NocW99}.

\begin{definition}
Let $b_\star$ be an outermost point of $\fovA$ 
such that the radius of curvature $\tilde r$ of~$\bfovA$ is well defined at $b_\star$.
Then the \emph{normalized radius of curvature of~$\bfovA$ at~$b_\star$} 
is~\mbox{$\mu \coloneqq \tilde r/ \numr \in [0,1]$}.
\end{definition}

Note that if $\mu = 0$, $b_\star$ is a corner of $\fovA$.
If $\mu = 1$, then near~$b_\star$, $\bfovA$ is well approximated by 
an arc of the circle with radius $\numr$ centered at the origin.
We show that the local convergence
is precisely determined by the value of $\mu$ at~$b_\star$.
In the upcoming analysis we use the following assumptions.

\begin{assumption}
\label{asm:wlog}
We assume that $\numr > 0$ and that it is attained at a non-corner $b_\star \in \bfovA$ with $\Arg(b_\star) = 0$.
\end{assumption}

\cref{asm:wlog} is essentially without any loss of generality.
Assuming $\numr > 0$ is trivial, as it only excludes $A=0$.  Since we are concerned
with finding the local rate of convergence at $b_\star$,
its location does not matter, and so we can assume a convenient one, that $b_\star$ 
is on the positive real axis.
As will be seen, our analysis does not lose any generality by assuming 
that $b_\star$ is not a corner.

\begin{assumption}
\label{asm:c2}
We assume that the current approximation $\mathcal{G}_j$  has been constructed using the supporting hyperplane $L_0$ 
passing through~$b_\star$, and so $b_\star \in \mathcal{Z}_j$,
and that~$\bfovA$ is twice continuously differentiable at $b_\star$.
\end{assumption}

\cref{asm:c2} is also quite mild.
Although by~\cref{asm:wlog} we assume that the outermost point~$b_\star$ is not a corner,
note that if $b_\star$ were (so $\mu = 0$),
then~\mbox{$b_\star \in \mathcal{Z}_j$} generally holds quite early in Uhlig's method
due to the fact that there are an 
infinite number of supporting hyperplanes passing through $b_\star$.
Returning to our assumption that~$b_\star$ is a not a corner ($\mu > 0$), 
it is true that Uhlig's method may sometimes only encounter the supporting 
hyperplane for $b_\star$ in the limit as his method converges.
However, via leveraging local optimization
we can modify Uhlig's algorithm so that~$b_\star$ is quickly and cheaply found
and used to update~$\mathcal{G}_j$ and~$\mathcal{Z}_j$;
we explain how and why this works in more detail in the first paragraph of~\cref{sec:alg2}.
Since such a modification does not alter how Uhlig's cuts are subsequently determined, and its cost is negligible,
it is quite informative to analyze his cutting procedure when~\cref{asm:c2} does hold.
Moreover, in a cutting-plane method, knowing an outermost point $b_\star$ and its associated hyperplane
does not guarantee convergence anyway.
Instead, Uhlig's method only terminates once~$\eps_k$
is sufficiently small for some~$k$,
which means that its cost is generally determined by how quickly 
it can sufficiently approximate~$\bfovA$ in \emph{neighborhoods about the outermost points};
per \cref{sec:background}, when $A$ is not a disk matrix, there can be up to $n$ such points.
The smoothness assumption ensures that there exists a unique osculating circle of $\bfovA$ at $b_\star$,
and consequently, 
the disagreement of the osculating circle and~$\bfovA$ decays at least cubicly as~$b_\star$ is approached;
for more on osculation, see, \eg \cite[chapter~2]{Kuh15}.

\begin{keyremark}
\label{key:why}
By our assumptions, $b_\star \in \mathcal{Z}_j$, and $\bfovA$
is twice continuously differentiable and has normalized radius of curvature~$\mu > 0$ at $b_\star$.
Since $\bd \mathcal{G}_j$ is a piecewise linear approximation of $\bfovA$,
the local behavior of a cutting-plane method 
is determined by the resulting second-order 
approximation errors, with
the higher-order errors being negligible. 
As $\bfovA$ is curved at $b_\star$, these second-order errors must be non-zero on both sides of~$b_\star$.
Now recall that the osculating circle of~$\bfovA$ at $b_\star$ locally agrees with~$\bfovA$ to at least second order.
Hence, near~$b_\star$, the second-order errors of a cutting-plane method applied to $A$
 \emph{are identical} to the second-order errors of applying the method to 
 a matrix $M$ with $\bfovX{M}$ being the same as that osculating circle.
Thus, to understand the local convergence rate  
of a cutting-plane method for general matrices, it actually suffices to study how the method
behaves on~$M$.
\end{keyremark}

We now define our key example $M$ such that, via two real parameters,
$\bfovX{M}$ is the osculating circle at the outermost point $b_\star$

\begin{figure}
\centering
\subfloat[Iteration $j$.]{
\resizebox*{6.0cm}{!}{\includegraphics[trim=1.3cm 0cm 1.3cm 0cm,clip]{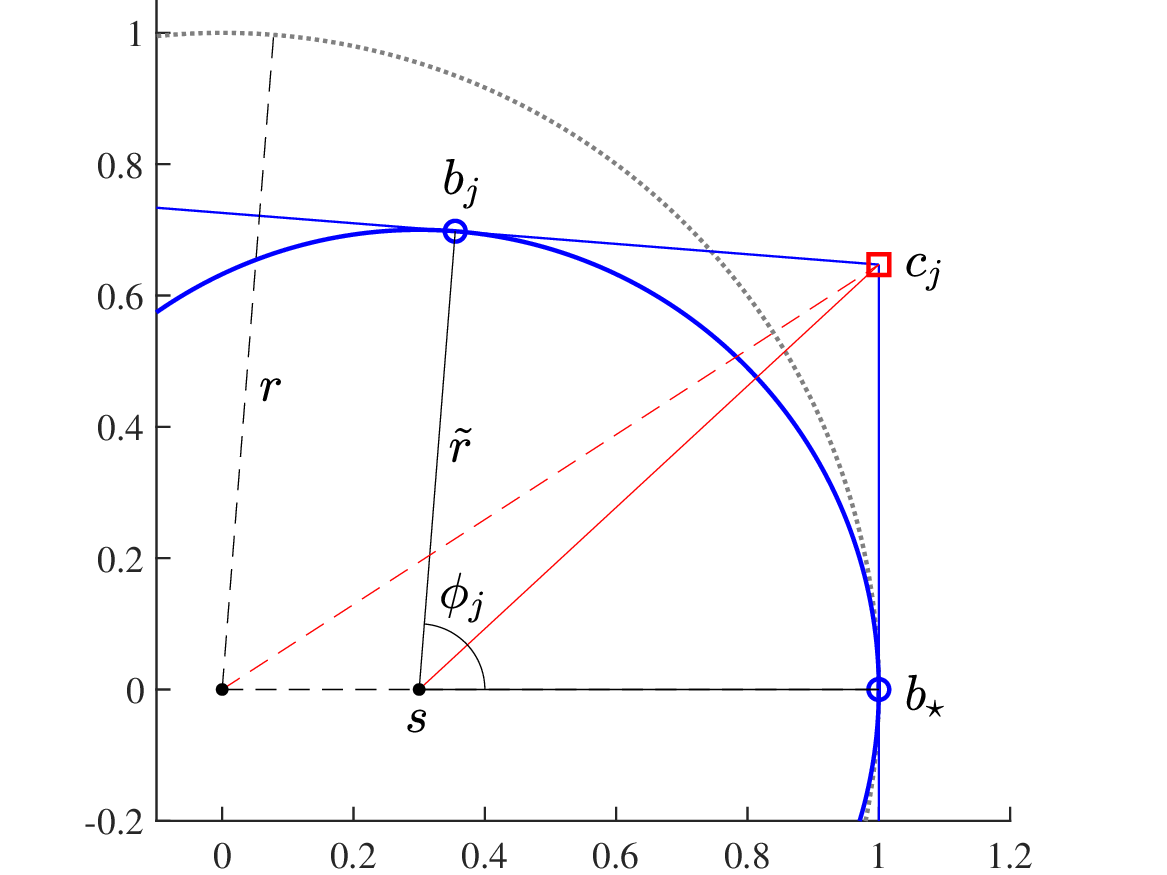}}
\label{fig:uhlig_j}
}
\subfloat[Iteration $j+1$.]{
\resizebox*{6.0cm}{!}{\includegraphics[trim=1.3cm 0cm 1.3cm 0cm,clip]{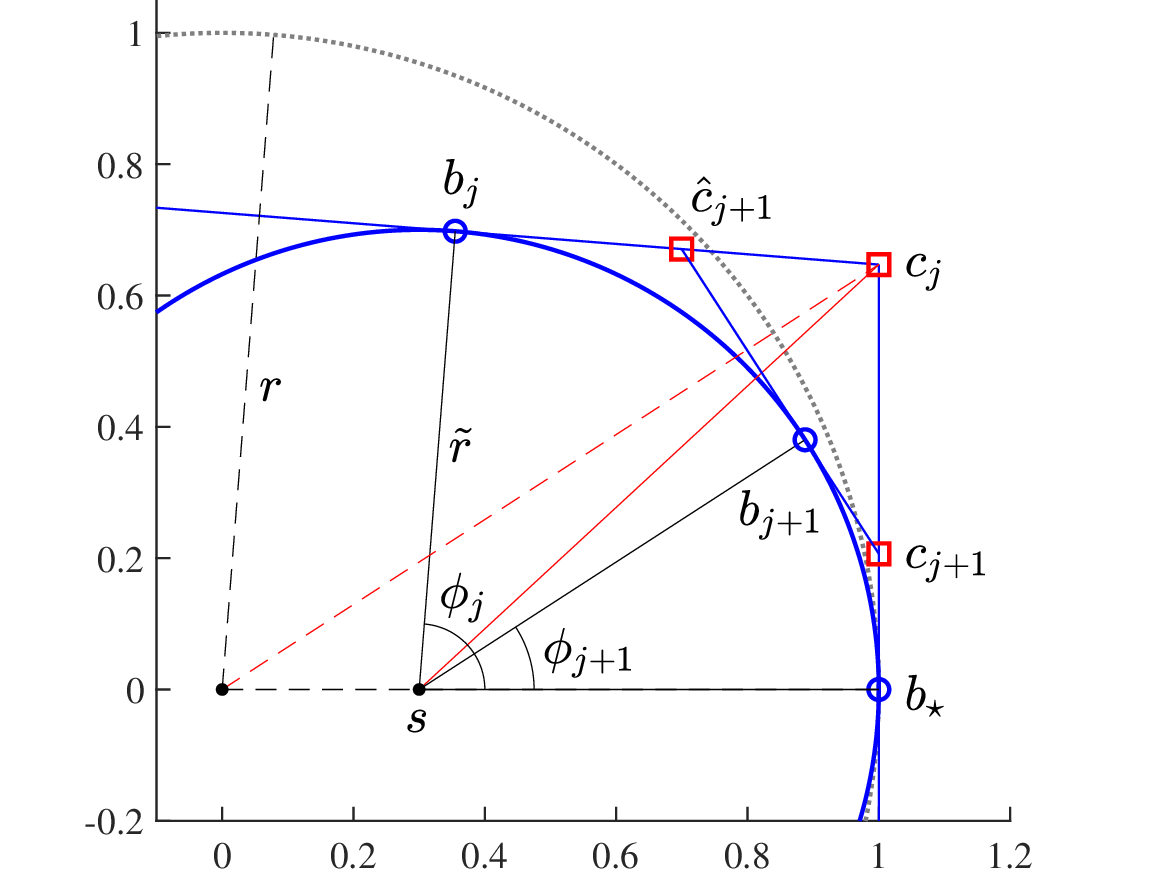}}
\label{fig:uhlig_jp1}
}
\caption{Depiction of Uhlig's cutting procedure on \cref{ex:circ} where~$z = 0.3$ and~\mbox{$\tilde r = 0.7$}.
The dotted circle is the circle of radius~$r(M) = 1$ centered at the origin.
See \cref{ex:circ} for a complete description of the plots.
}
 \label{fig:uhlig}
\end{figure}

\def\cm{s}

\begin{example}[See \cref{fig:uhlig} for a visual description]
\label{ex:circ}
For $n \geq 2$, let 
\begin{equation}
	\label{eq:model}
	M = \cm I + \tilde r K_n,
\end{equation}
where $K_n  \in \C^{n \times n}$ is any disk matrix with $r(K_n)=1$, e.g., \eqref{eq:crabb},
$\cm \geq 0$, and~$\tilde r > 0$.
Clearly, $\fovX{M}$ is a disk with radius $\tilde r$ centered at $\cm$  on the real axis with 
outermost point~$b_\star = \cm + \tilde r = r$
and $\mu = \tfrac{\tilde r}{r} > 0$ at $b_\star$, 
where~$r \coloneqq r(M)$, a shorthand we will often use in \cref{sec:rate}~and~\cref{sec:alg2}.
Thus, given any matrix~$A$ satisfying \cref{asm:wlog,asm:c2} with any $\mu \in (0,1]$,
by choosing $\cm \geq 0$ and~$\tilde r > 0$ appropriately
we have that~$\bfovX{M}$ agrees exactly with the osculating circle 
of~$\bfovA$ at~$b_\star$.  
Assume that~$\theta_\star = 0$ and~$\theta_j \in (-\pi,0)$
have been used to construct~$\mathcal{G}_j$ and~$\mathcal{Z}_j$, \ie
supporting hyperplanes~$L_{\theta_\star}$ and~$L_{\theta_j}$ respectively pass through 
boundary points~$b_\star$ and~$b_j = \cm + \tilde r \eix{\phi_j}$ of $\fovX{M}$,
where~\mbox{$\phi_j \coloneqq \Arg (b_j - \cm) \in (0,\pi)$}.
Therefore, $b_\star = z_0 \in \mathcal{Z}_j$ and since $\fovX{M}$ is a disk, 
it also follows that $b_j=z_{\theta_j} \in \mathcal{Z}_j$, where~$\theta_j = -\phi_j$.
Further assume that for all~\mbox{$\theta \in (\theta_j,\theta_\star)$}, $z_\theta \not\in \mathcal{Z}_j$,
and so~$c_j \coloneqq L_{\theta_\star} \cap L_{\theta_j}$ is a corner of~$\mathcal{G}_j$.
Now suppose~$c_j$ is cut, by any cutting procedure, Uhlig's  or otherwise.
This results in a boundary point~$b_{j+1}$ of $\fovX{M}$ being added to $\mathcal{Z}_j$ 
and $c_j$ is replaced by two new corners, $\hat c_{j+1}$ and~$c_{j+1}$,
that respectively lie on~$L_{\theta_j}$ and $L_{\theta_\star}$; 
due to their orientation with respect to~$b_{j+1}$, we refer to 
$\hat c_{j+1}$ as a counter-clockwise (CCW) corner and~$c_{j+1}$ as a clockwise~(CW) corner.
If~$c_{j+1}$ is then cut next, this produces two more corners,~$\hat c_{j+2}$
and~$c_{j+2}$, with $c_{j+2}$ also on~$L_{\theta_\star}$ and between $c_{j+1}$ and~$b_\star$.
Note that if~$\{\phi_k\} \to 0$, then the sequence of CW corners $\{c_k\} = c_j,c_{j+1},c_{j+2},\ldots$ 
converges to $b_\star$.
To understand the local behavior of a cutting-plane technique, we will
analyze~$\{\phi_k\}$ and~$\{c_k\}$, \ie the case when the cuts are applied 
to the CW corners that are sequentially generated~on~$L_{\theta_\star}$.
\end{example}

Some remarks on \cref{ex:circ} are in order,
as it ignores the CCW corners  $\hat c_j$,
and cutting these CCW corners may also introduce new corners that need to be cut.
However, since $\{c_k\}$ is a subsequence of all the corners generated
by Uhlig's method 
to sufficiently approximate $\bfovX{M}$ between boundary points~$b_\star$ and~$b_j$,
analyzing $\{c_k\}$ gives a lower bound on its local efficiency.  
Furthermore, this often describes the true local efficiency because, as will become clear, 
for many problems, there are either no or few CCW corners 
that requiring cutting.
Finally, in \cref{sec:alg2}, we introduce an improved cutting scheme
that guarantees only CW corners must be cut.

\begin{lemma}
\label{lem:uhlig_cuts}
Recalling that $\phi_j \coloneqq \Arg (b_j - \cm)$, 
if Uhlig's cuts are sequentially applied to the corners $\{c_k\}$ described in \cref{ex:circ},
then for all $k \geq j$, 
\beq
	\label{eq:uhlig_angle}
	\phi_{k+1} 
	= \arctan \big( \mu \tan \tfrac{1}{2}\phi_k \big).
\eeq
\end{lemma}
\begin{proof}
Since $\fovX{M}$ is a disk with tangents at $b_k$ and $b_\star$ determining $c_k$,
first note that we have that $\Arg(c_k - \cm) = \tfrac{1}{2}\phi_k$.
Then \mbox{$\tan \tfrac{1}{2} \phi_k = \tilde r^{-1}|c_k - b_\star|$},
and since the tangent at $b_\star$ is vertical, we also have that
$\Arg(c_k) = \arctan(r^{-1} |c_k - b_\star|)$.
Thus via substitution,
\[	
	\Arg(c_k) 
	= \arctan ( r^{-1} \tilde r \tan \tfrac{1}{2}\phi_k ) 
	= \arctan ( \mu \tan \tfrac{1}{2}\phi_k ).
\]
The proof is completed since $\phi_{k+1} = \Arg (b_{k+1} - \cm)$
is also equal to $\Arg(c_k)$.
\end{proof}

\begin{theorem}
\label{thm:uhlig_angle}
The sequence $\{\phi_k\}$ produced by Uhlig's 
cutting procedure and described by recursion~\eqref{eq:uhlig_angle}
 converges to zero Q-linearly with rate $\tfrac{1}{2} \mu$.
\end{theorem}
\begin{proof}
First note that $\lim_{k\to\infty} \phi_k = 0 = \phi_\star$ and $\phi_k > 0$ for all $k \geq j$.  Then
\[
	\lim_{k \to \infty}
	\frac{| \phi_{k+1} - \phi_\star |}{| \phi_{k} - \phi_\star |} 
	= \lim_{k \to \infty} \frac{\phi_{k+1}}{\phi_{k}} 
	= \lim_{k \to \infty} \frac{\arctan \left( \mu \tan \tfrac{1}{2} \phi_k \right)}{\phi_k}.
\]
Since the numerator and denominator both go to zero as $k\to\infty$, the result 
follows by considering the continuous version of the limit:
\[
	\lim_{\phi \to 0} \frac{\arctan \big( \mu \tan \tfrac{\phi}{2} \big)}{\phi}
	= \lim_{\phi \to 0} \frac{\mu \tan \tfrac{1}{2}\phi}{\phi}
	= \lim_{\phi \to 0} \frac{\tfrac{1}{2}\mu \phi }{\phi}
	= \tfrac{1}{2} \mu,
\]
where the first and second equalities are
obtained, respectively, using small-angle approximations 
$\arctan x \approx x$ and $\tan x \approx x$ for $x\approx 0$.
\end{proof}

While \cref{thm:uhlig_angle} tells us how quickly $\{\phi_k\}$ will converge,
we really want to estimate how quickly the error $\eps_j$ becomes sufficiently small.
For that, we must consider how fast the moduli of the corresponding outermost corners $c_k$
converge.

\begin{theorem}
\label{thm:uhlig_ub}
Given the sequence 
$\{\phi_k\}$ from \cref{thm:uhlig_angle},
the corresponding sequence~$\{|c_k|\}$ 
converges to $r$ Q-linearly with rate $\tfrac{1}{4} \mu^2$.
\end{theorem}
\begin{proof}
First note that 
\[
	\cos \phi_{k+1} = \frac{r}{|c_k|}
	\qquad \text{and so} \qquad
	|c_k| = r \sec \phi_{k+1} > r
\]
for all $k\geq j$.
Thus, we consider the limit
\[
	\lim_{k \to \infty} \frac{| |c_{k+1}| - r |}{| |c_k| - r |} 
	= \lim_{k \to \infty} \frac{r \sec \phi_{k+2} - r}{r \sec \phi_{k+1} - r} 
	= \lim_{k \to \infty} \frac{\sec \phi_{k+1} - 1}{\sec \phi_k - 1},
\]
which when substituting in $\phi_{k+1} = \arctan \big( \mu \tan \tfrac{1}{2}\phi_k \big)$ becomes
\[
	\lim_{k \to \infty} 
	\frac{\sec \left( \arctan \left( \mu \tan \tfrac{1}{2}\phi_k \right)\right) - 1}{\sec \phi_k - 1}.
\]
Since the numerator and denominator both go to zero as $k\to\infty$, we 
consider the continuous version of the limit, \ie
\[
	\lim_{\phi \to 0} 
	\frac{\sec \left( \arctan \left( \mu \tan \tfrac{1}{2}\phi \right)\right) - 1}{\sec \phi - 1} 
	= \lim_{\phi \to 0} 
	\frac{\sec \left( \mu \tan \tfrac{1}{2}\phi \right) - 1}{\sec \phi - 1} 
	= \lim_{\phi \to 0} 
	\frac{\sec \left( \tfrac{1}{2}\mu \phi \right) - 1}{\sec \phi - 1},
\]
again using the small-angle approximations for $\arctan$ and $\tan$.
Letting $g(\phi) = \sec(\phi)$, 
and noting that $g(0) = 1$,  $g^\prime(0) = 0$, and $g^{\prime\prime}(0) = 1$,
its Taylor expansion about $0$ is 
\[ 
	g(\phi) 
	= 1 + \tfrac{1}{2}\phi^2 + \sum_{n = 3}^\infty \frac{g^{(n)}(0)}{n!} \phi^n.
\]
Replacing the secant terms in both the numerator and denominator of the limit above
with their Taylor expansions,
we obtain the equivalent limit yielding our result:
\[
	\lim_{\phi \to 0} 
	\frac{\tfrac{1}{2}(\tfrac{1}{2}\mu\phi)^2 + \bigO(\phi^3)}{\tfrac{1}{2}\phi^2 + \bigO(\phi^3)}
	= \tfrac{1}{4}\mu^2.
\]
\end{proof}

As \cref{ex:circ} can model any $\mu \in (0,1]$, 
per \cref{key:why}, \cref{thm:uhlig_angle,thm:uhlig_ub} also accurately describe 
the local behavior of Uhlig's cutting procedure at outermost points in $\fovA$ 
for any matrix $A$.  
Moreover, due to the squaring and one-quarter factor in \cref{thm:uhlig_ub}, 
the linear rate of convergence becomes very fast rather rapidly 
as~$\mu$ decreases from one,
ultimately becoming superlinear if the outermost point is a corner ($\mu = 0$).
We can also estimate the cost of approximating 
 $\bfovA$ about~$b_\star$, determining how many iterations will be needed
until it is no longer necessary to refine corner $c_k$, \ie
the value of $k$ such that $|c_k| \leq \numr \cdot ( 1 + \tau_\mathrm{tol} )$.
For simplicity, it will now be more convenient to assume that~$j=0$
with $|c_0| = \beta\numr$ for some scalar~$\beta > (1 + \tau_\mathrm{tol})$.
Via the Q-linear rate given by \cref{thm:uhlig_ub}, we have that 
\[
	|c_k| - \numr \leq (|c_0| - \numr) \cdot \left(\tfrac{1}{4}\mu^2\right)^k,
\]
and so if 
\[
	\numr + (|c_0| - \numr) \cdot \left(\tfrac{1}{4}\mu^2\right)^k \leq  \numr \cdot ( 1 + \tau_\mathrm{tol} ),
\]
then it follows that $|c_k| \leq \numr \cdot ( 1 + \tau_\mathrm{tol} )$, \ie it does not need to be refined further.
By first dividing the above equation by $\numr$ and doing some simple manipulations,
we have that $|c_k|$ is indeed sufficiently close to $\numr$~if
\beq
	k \geq \frac{\log (\tau_\mathrm{tol}) - \log (\beta - 1) }{\log \left(\tfrac{1}{4}\mu^2\right)}.
\eeq
Using \cref{ex:circ} with $\beta = 100$, and $\tau_\mathrm{tol} = \texttt{1e-14}$, only
$k \approx 27$, $14$, $7$, and $4$ iterations are needed, respectively, 
for $\mu = 1$, $0.5$, $0.1$, and $0.01$.
This is indeed rather fast for linear convergence.
Of course, if $\fovA$ has more than one outermost point, the total cost of a cutting-plane method
increases commensurately, since $\bfovA$ must be well approximated about all of these outermost points.
For disk matrices, all boundary points are outermost, and so the cost blows up, per \cref{thm:uhlig_disk}.

\section{An improved cutting-plane algorithm}
\label{sec:alg2}
We now address some inefficiencies in Uhlig's method
by giving an improved cutting-plane method.
The two main components of this refined algorithm are as follows.
First, any of the local optimization techniques from \cref{sec:alg1}
also allows us to more efficiently locate outermost points in~$\fovA$.
This is possible because each outermost point is bracketed on~$\bfovA$ by 
two boundary points of $\fovA$ in $\mathcal{Z}_j$,
and these brackets improve as $\mathcal{G}_j$ more accurately approximates $\fovA$.
Therefore, once $\mathcal{G}_j$ is no longer a crude approximation, 
these brackets can be used to initialize optimization to find global maximizers of~$h(\theta)$, 
and thus, globally outermost points of $\fovA$.
Second, given a boundary point of $\fovA$ that is also known to be locally outermost,
we use a new cutting procedure
that reduces the total number of cuts needed to sufficiently approximate $\bfovA$
in this region.
When this new cut cannot be invoked, we will fall back on Uhlig's cutting procedure.
In the next three subsections, we describe our new cutting strategy, establish a Q-linear rate of convergence for it, 
and finally, show how these cuts can be sufficiently well estimated so that our theoretical convergence rate result
is indeed realized in practice.  Finally, pseudocode of our completed algorithm is given in~\cref{alg2}.

\subsection{An optimal-cut strategy}
\label{sec:opt_cut}
Again consider \cref{ex:circ}.
In \cref{fig:uhlig_jp1}, Uhlig's cut of corner $c_j$ between $b_j$ (with $|b_j| < r$) and $b_\star$ produces
two new corners~$\hat c_{j+1}$ and $c_{j+1}$, but since $|\hat c_{j+1}| < r$ and $|c_{j+1}| > r$,
it is only necessary to subsequently refine $c_{j+1}$.  
However, in \cref{fig:uhlig_two} we show another scenario where both of the two new corners
produced by Uhlig's cut will require subsequent cutting as well.  
While \cref{thm:uhlig_angle,thm:uhlig_ub} indicate
the number of iterations Uhlig's method needs to sufficiently refine the sequence $\{c_k\}$,
they do not take into account that the CCW corners that are generated may also need to be cut.
Thus, the total number of eigenvalue computations with $H(\theta)$ 
can be higher than what is suggested by these two theorems.
However, comparing \cref{fig:uhlig_jp1,fig:uhlig_two} immediately 
suggests a better strategy, namely, to make the largest reduction in the angle $\phi_j$
such that the CCW corner~$\hat c_{j+1}$ (between $b_j$ and $c_j$ on the tangent line for $b_j$) 
does not subsequently need to be refined, \ie such that $|\hat c_{j+1}| = r$. 
In \cref{fig:uhlig_two}, this ideal corner is labeled $d_j$, while the corresponding optimal cut for this same example
is shown in \cref{fig:opt_cut}, where $d_j$ coincides with $\hat c_{j+1}$, and so the latter is not labeled.

\subsection{Convergence analysis of the optimal cut}
\label{sec:opt_conv}
Before describing how to compute optimal cuts,
we  derive the convergence rate of the sequence of angles~$\{\phi_k\}$
this strategy produces.
Per \cref{key:why}, it again suffices to study \cref{ex:circ}.

\begin{figure}[t]
\centering
\subfloat[Uhlig's cut.]{
\resizebox*{6.0cm}{!}{\includegraphics[trim=1.3cm 0cm 1.3cm 0cm,clip]{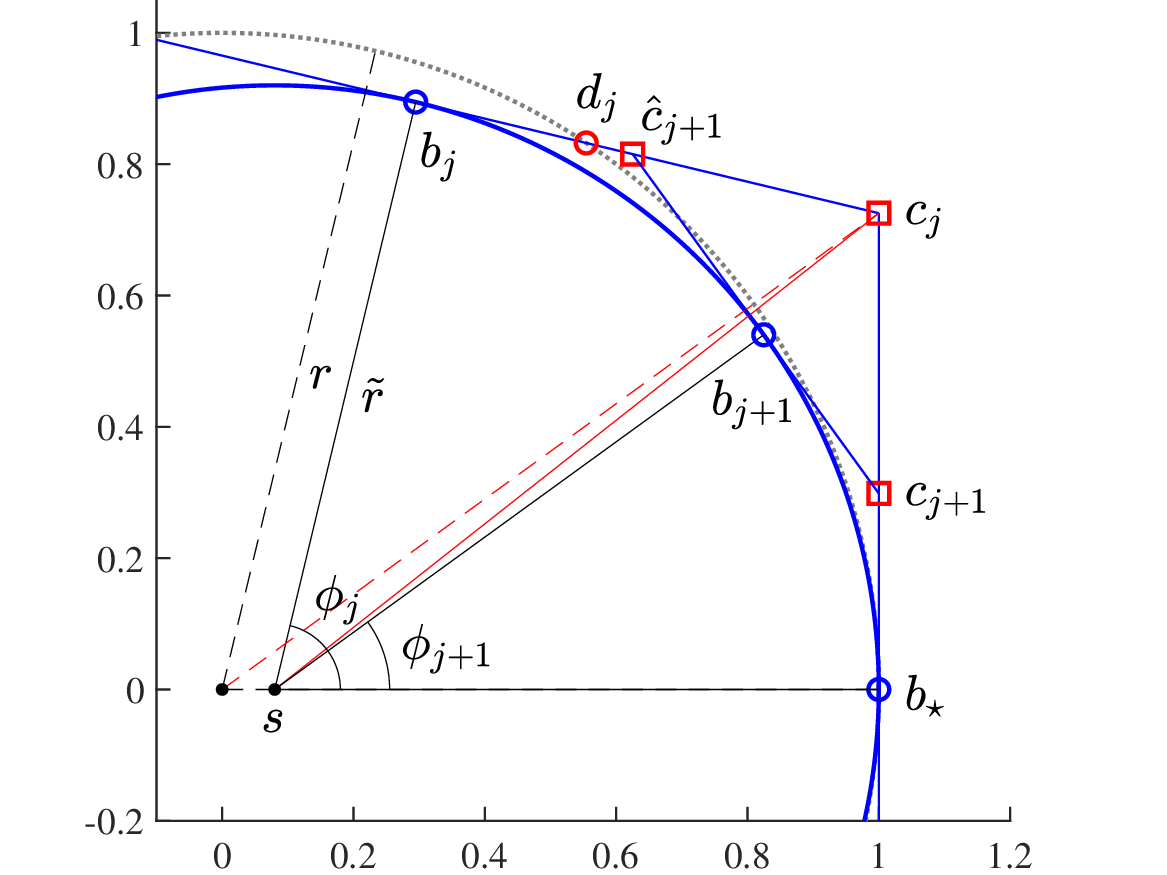}}
\label{fig:uhlig_two}
}
\subfloat[The optimal cut.]{
\resizebox*{6.0cm}{!}{\includegraphics[trim=1.3cm 0cm 1.3cm 0cm,clip]{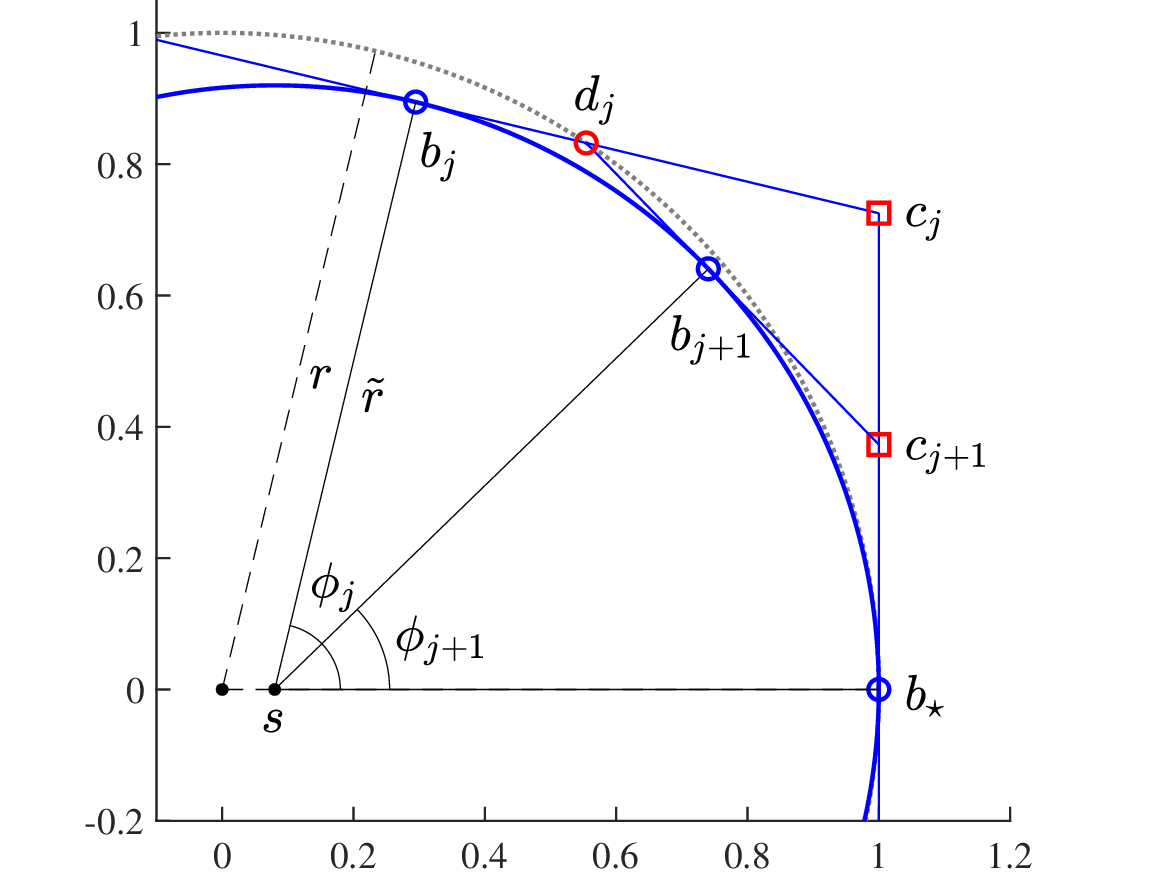}}
\label{fig:opt_cut}
}
\caption{Depictions of corner $c_j$ between boundary points $b_j$ and $b_\star$
being cut by Uhlig's cutting procedure (left) and the optimal cut (right),
where the latter always makes the largest possible reduction in $\phi_j$ such that
only corner $c_{j+1}$ must be refined.
}
 \label{fig:cuts_comp}
\end{figure}

\begin{lemma}
\label{lem:dj}
Given \cref{ex:circ},
additionally assume 
that $|b_j| < r$ and $\tilde r < r$, and so at $b_\star$, 
$\mu \in (0,1)$.
Then the point on~$L_{\theta_j}$,  \ie
the supporting hyperplane passing through $b_j$ where $\theta_j = -\phi_j \in (-\pi,0)$,
that is closest to $c_j$ and has modulus $r$ is
\begin{subequations}
	\begin{align}
	\label{eq:dlem}
	d_j &= b_j - \imagunit t_j \eix{\phi_j}, \quad \text{where} \\
	\label{eq:tlem}
	t_j &= -\cm \sin \phi_j + \sqrt{\cm^2 \sin^2 \phi_j + 2\cm \tilde r(1 - \cos \phi_j)} > 0.
	\end{align}
\end{subequations}
\end{lemma}
\begin{proof}
As  $\phi_j \in (0,\pi)$, 
clearly \eqref{eq:dlem} must hold for some~$t_j > 0$.
To obtain~\eqref{eq:tlem}, we use the fact that $|d_j|^2 = r^2$ and solve for $t_j$ using \eqref{eq:dlem}.
Setting~$u = \eix{\phi_j}$, we have
\begin{align*}
	0 &= |d_j|^2 - r^2 = (b_j - \imagunit t_j u) \big(\overline b_j + \imagunit t_j \overline{u}) - r^2 
	= t_j^2 + \imagunit (b_j \overline{u} - \overline b_j u)t_j + |b_j|^2  - r^2.
\end{align*}
which by substituting in the following two equivalences
\begin{align*}
	b_j \overline{u} - \overline b_j u
	&= (\cm + \tilde r  u) \overline{u} - (\cm + \tilde r  \overline{u}) u
	= \cm (\overline{u} - u) = -\imagunit 2\cm \sin \phi_j, \\
	|b_j|^2 &= \cm^2 + \tilde r^2 + \cm \tilde r(u + \overline{u}) = \cm^2 + \tilde r^2 + 2\cm \tilde r \cos \phi_j,
\end{align*}
yields
\[
	0 = t_j^2 + (2 \cm \sin \phi_j) t_j + (\cm^2 + \tilde r^2 - r^2 + 2\cm\tilde r \cos \phi_j).
\]
By substituting in $r^2 = (\cm + \tilde r)^2 = \cm^2 + \tilde r^2 + 2\cm\tilde r$, this simplifies further to
\[
	0 = t_j^2 + (2 \cm \sin \phi_j) t_j + 2\cm \tilde r (\cos \phi_j - 1).
\]
Thus, by the quadratic formula, we obtain \eqref{eq:tlem}.
\end{proof}

\begin{lemma}
\label{lem:opt_angle}
Given the assumptions in \cref{lem:dj} and  $t_j > 0$ from \eqref{eq:tlem}, 
also suppose that \mbox{$\phi_j \in (0,\tfrac{\pi}{2})$}.
Then if optimal cuts are sequentially applied to the corners~$\{c_k\}$ described in \cref{ex:circ},
for all $k \geq j$, 
\beq
	\label{eq:opt_thetajp1}
	\phi_{k+1} 
	= -\phi_k + 2\arctan \left( \frac{\tilde r \sin \phi_k - t_k \cos \phi_k}{\tilde r \cos \phi_k + t_k\sin\phi_k} \right).
\eeq
\end{lemma}
\begin{proof}
Let $\hat \phi = \phi_k - \Arg(d_k - \cm)$.  Then it follows that 
\begin{equation}
	\label{eq:theta_jp1_proof}
	\phi_{k+1} = \phi_k - 2\hat \phi 
	= -\phi_k + 2\Arg(d_k - \cm)  
	= - \phi_k + 2\arctan \left( \frac{\Im (d_k - \cm)}{\Re (d_k - \cm)} \right),
\end{equation}
where the last equality follows because $\Re (d_k - \cm) > 0$, as $\Re b_k > \cm$.
Using \eqref{eq:dlem} and substituting in $b_k = \cm + \tilde r\e^{\imagunit \phi_k}$, 
we have that 
\[
	d_k - \cm = \tilde r \e^{\imagunit \phi_k} - \imagunit t_k \e^{\imagunit \phi_k}
	\qquad \Longleftrightarrow \qquad
	\begin{aligned}
	\Re (d_k - \cm) &= \tilde r \cos \phi_k + t_k \sin \phi_k,\\
	\Im (d_k - \cm) &= \tilde r \sin \phi_k - t_k \cos \phi_k.
	\end{aligned}
\]
Substituting these into~\eqref{eq:theta_jp1_proof} completes the proof.
\end{proof}

Before deriving how fast $\{\phi_k\}$ converges, we show that it indeed converges to zero.

\begin{lemma}
\label{lem:opt_zero}
Given the assumptions of \cref{lem:opt_angle},
the recursion \eqref{eq:opt_thetajp1} for
optimal cuts produces a sequence of angles $\{\phi_k\}$ converging to zero.
\end{lemma}
\begin{proof}
By construction, $\{\phi_k\}$ is monotone, \ie $\phi_{k+1} < \phi_k$ for all $k \geq j$,
and bounded below by zero, and so $\{\phi_k\}$ converges to some limit $l$.
Per \eqref{eq:tlem}, the~$t_k$ values appearing in \eqref{eq:opt_thetajp1} depend on $\phi_k$,
so we define the analogous continuous function
\beq
	\label{eq:t_cont}
	t(\phi) = -\cm \sin \phi + \sqrt{\cm^2 \sin^2 \phi + 2\cm \tilde r (1 - \cos \phi )}.
\eeq
Now by way of contradiction, assume that $l > 0$ and so $0 < l < \phi_j < \tfrac{\pi}{2}$.
Thus,
\[
	\lim_{k\to\infty} \phi_{k+1} = l = -l + 2\arctan \left( \frac{\tilde r \sin l - t(l) \cos l}{\tilde r \cos l + t(l) \sin l} \right)
	\ \Leftrightarrow \
	\tan l = \frac{\tilde r \sin l - t(l) \cos l}{\tilde r \cos l + t(l) \sin l}.
\]
Then, by multiplying both sides by $\tilde r \cos l + t(l) \sin l$ and rearranging terms,
we obtain the equality \mbox{$t(l)(\sin^2 l + \cos^2 l) = 0$}, and so $t(l) = 0$.
However, \cref{lem:dj} states that $t(l) > 0$ should hold since $l \in (0,\tfrac{\pi}{2})$,
a contradiction, and so $l=0$.
\end{proof}

We now have the necessary pieces to derive the exact rate of convergence of 
the angles produced by optimal cuts.

\begin{theorem}
\label{thm:opt_conv}
The sequence $\{\phi_k\}$ produced by optimal cuts 
and described by recursion \eqref{eq:opt_thetajp1} 
converges to zero Q-linearly with rate $\tfrac{2(1 - \sqrt{1 - \mu})}{\mu} - 1$.
\end{theorem}
\begin{proof}
By \cref{lem:opt_angle,lem:opt_zero}, 
\eqref{eq:opt_thetajp1} holds, $\phi_k \to \phi_\star = 0$, and $\phi_k \geq 0$ for all~$k \geq j$, so
\[
	\lim_{k \to \infty} \frac{|\phi_{k+1} - \phi_\star|}{| \phi_{k} - \phi_\star |}  
	= \lim_{k \to \infty} \frac{\phi_{k+1}}{\phi_{k}} 
	= \lim_{k \to \infty} \frac{ 
		-\phi_k + 2\arctan \left( \frac{\tilde r \sin \phi_k - t_k \cos \phi_k}{\tilde r \cos \phi_k + t_k \sin\phi_k} \right)
	}{\phi_k}.
\]
Using the continuous version of $t_k$ given in \eqref{eq:t_cont},
we instead consider the entire limit in continuous form:
\beq
	\label{eq:opt_lim_cont}
	\lim_{\phi \to 0} \frac{ 
		-\phi + 2\arctan \left( \frac{\tilde r \sin \phi - t(\phi) \cos \phi}{\tilde r \cos \phi + t(\phi) \sin\phi} \right)
	}{\phi}
	= -1 + \lim_{\phi \to 0}
		\frac{2}{\phi} \cdot
		\frac{ \tilde r \phi - t(\phi)\cos \phi}{\tilde r \cos\phi + t(\phi) \phi},
\eeq
where the equality holds by using the small-angle approximations $\arctan x \approx x$ 
(as the ratio inside the $\arctan$ above goes to zero as $\phi \to 0$)
and $\sin x \approx x$.
Again using 
$\sin x \approx x$ as well as the small-angle approximation $1 - \cos x \approx \tfrac{1}{2}x^2$,
we also have the small-angle approximation 
\beq
	\label{eq:t_approx}
	t(\phi) 
	\approx -\cm \phi + \sqrt{\cm^2 \phi^2 + 2\cm\tilde r \big(\tfrac{1}{2}\phi^2\big)} 
	= -\phi \left(\cm - \phi \sqrt{\cm^2 + \cm\tilde r}\right) 
	= -\phi \left(\cm - \sqrt{\cm r} \right),
\eeq
where the last equality holds since $\tilde r = r - \cm$.
Via substituting in \eqref{eq:t_approx}, the limit on the right-hand side of \eqref{eq:opt_lim_cont} is
\[
	\lim_{\phi \to 0}
		\frac{2}{\phi}\cdot
		\frac{\tilde r \phi + \phi \left(\cm - \sqrt{\cm r} \right) \cos \phi
		}{\tilde r \cos \phi - \phi \left(\cm - \sqrt{\cm r} \right) \phi } 
	= \lim_{\phi \to 0}
		\frac{2 \left( \tilde r  + \left(\cm - \sqrt{\cm r} \right) \cos \phi \right)
		}{\tilde r \cos \phi - \phi^2 \left(\cm - \sqrt{\cm r} \right) } \\	
	= \frac{
			2 \left( \tilde r  + \cm - \sqrt{\cm r} \right)
		}{\tilde r}.
\]
Recalling that $\tilde r = \mu r$ and that $\cm = r - \tilde r = r - \mu r$,
by substitutions 
we can rewrite the ratio above as
\[
	\frac{2\left(\mu r + (r - \mu r) - \sqrt{(r - \mu r) r} \right) }{\mu r}
	= \frac{2\left(r - r \sqrt{1 - \mu} \right) }{\mu r}
	= \frac{2\left(1 - \sqrt{1 - \mu} \right) }{\mu}.
\]
Subtracting one from the value above completes the proof.
\end{proof}

As we show momentarily, optimal cuts have a total lower cost than Uhlig's cutting procedure.
Thus, there is no need to derive an analogue of \cref{thm:uhlig_ub} for describing
 the convergence rate of  the moduli of corners~$c_k$ produced by the optimal-cut strategy.

\subsection{Computing the optimal cut}
\label{sec:compute_opt}
Suppose that $b_\star \in \mathcal{Z}_j$ attains the value of~$l_j$,
and that $b_\star$ is also locally outermost in~$\fovA$,
and let \mbox{$\gamma = |b_\star| \leq \numr$}.
Without loss of generality, we assume that $\Arg(b_\star) = 0$,
and let $b_j \in \mathcal{Z}_j$ be the next known boundary point of $\fovA$ with $\Arg(b_j) \in (0,\pi)$.
We can model $\bfovA$ between~$b_\star$ and~$b_j$ by fitting a quadratic 
that interpolates $\bfovA$ at~$b_\star$ and~$b_j$.
If this model is a good fit, then it can be used to estimate~$d_j$,
and thus, also the optimal cut. 

Since $b_\star$ is also a locally outermost point of $\fovA$ and $\Arg(b_\star) = 0$, 
we can interpolate these boundary points using the 
sideways quadratic (opening up to the left in the complex plane)
\[
	q(y) = q_2 y^2 + q_1 y + q_0,
\]
with the remaining degree of freedom used to specify 
that $q(y)$ should be tangent to~$\fovA$ at $b_\star$.
Clearly, $q(y)$ cannot be a good fit if $L_{\theta_j}$,
the supporting hyperplane passing through $b_j$, is increasing from left to right in the complex plane;
hence, we also assume that \mbox{$\theta_j \in (-\tfrac{\pi}{2},0)$}.
 Let~\mbox{$\theta_\dagger \in (\theta_j,0)$} denote the angle of the supporting hyperplane
 for the optimal cut, \eg for \cref{ex:circ}, the one that passes through $d_j$ and the boundary point $b_{j+1}$ 
between $b_j$ and $b_\star$.
By our criteria, the equations
\beq
	q(0) = \gamma, 
	\quad
	q(\Im b_j) = \Re b_j, 
	\quad \text{and} \quad
	q^\prime(0) = 0
\eeq
determine the coefficients $q_0$, $q_1$, and $q_2$, and solving these yields 
\beq
	q_2 = \frac{ \Re b_j - \gamma}{(\Im b_j)^2},
	\quad
	q_1 = 0,
	\quad \text{and} \quad
	q_0 = \gamma.
\eeq

\begin{algfloat}[t]
\begin{algorithm}[H]
\floatname{algorithm}{Algorithm}
\caption{An Improved Cutting-Plane Algorithm}
\label{alg2}
\begin{algorithmic}[1]
	\REQUIRE{  
		$A \in \C^{n \times n}$ with $n \geq 2$ and $\tau_\mathrm{tol} > 0$.
	}
	\ENSURE{ 
		$l$ such that $|l -\numr| \leq \tau_\mathrm{tol} \cdot \numr$.
		\\ \quad
	}

	\STATE $\mathcal{G} \gets P_{\theta_1} \cap \cdots \cap P_{\theta_4}$ where $\theta_\ell = \tfrac{\ell -1 }{2}\pi$
		for $\ell = 1,2,3,4$
	\STATE $\mathcal{Z} \gets  \{ z_{\theta_\ell} :  \ell = 1,2,3,4\}$ 
	\STATE $l \gets \max \{ |b| : b \in \mathcal{Z} \}$, $u \gets  \max \{ |c| : c \text{ a corner of } \mathcal{G} \}$
	\WHILE {$\tfrac{ u - l }{l} > \tau_\mathrm{tol}$ }
		\STATE $L_{\theta_1} \gets$ supporting hyperplane 
			for the boundary point in $\mathcal{Z}$ attaining~$l$
		\STATE $\gamma \gets$ local max of $\rho(H(\theta))$ via optimization initialized at $\theta_1$
		\STATE $\mathcal{G} \gets \mathcal{G} \cap 
			P_{\theta_1} \cap \cdots \cap P_{\theta_q}$
			for the $q$ angles $\theta_\ell$ encountered during optimization
		\STATE $\mathcal{Z} \gets \mathcal{Z} \cup  \{ z_{\theta_\ell} : \ell = 1,\ldots,q\}$
		\STATE $c \gets \text{ outermost corner of } \mathcal{G}$
		\IF { the optimal cut should be applied to $c$ per \cref{sec:compute_opt} }
			\STATE $\theta \gets$ angle $\theta_\dagger$ is given by \eqref{eq:opt_cut_angle} 
					(rotated and flipped as necessary)
		\ELSE
			\STATE $\theta \gets -\Arg(c)$ \COMMENT{Uhlig's cut}
		\ENDIF 
		\STATE $\mathcal{G} \gets \mathcal{G} \cap P_\theta$
		\STATE $\mathcal{Z} \gets \mathcal{Z} \cup  \{ z_{\theta}  \}$ 
		\STATE $l \gets \max \{ |b| : b \in \mathcal{Z} \}$, $u \gets  \max \{ |c| : c \text{ a corner of } \mathcal{G} \}$
	\ENDWHILE
\end{algorithmic}
\end{algorithm}
\vspace{-0.4cm}
\algnote{
For simplicity, we forgo describing pseudocode to exploit possible normality of~$A$ or symmetry of $\fovA$,
and assume that $A\neq0$,  eigenvalues and local maximizers are obtained exactly, 
optimization is monotonic, \ie $l \leq \gamma$ is guaranteed, 
and there are no ties for the boundary point in line~5.  
}
\end{algfloat}

We can assess whether $q(y)$ is a good fit for $\bfovA$ about $b_\star$ 
by checking how close $q(y)$ is to being tangent to $\bfovA$ at $b_j$, 
\ie $q(y)$ is a good fit if 
\beq
	q^\prime(\Im b_j) \approx \tan \theta_j.
\eeq
If these two values are not sufficiently close, then we consider $q(y)$ a poor local approximation of $\bfovA$ at $b_j$ (and $b_\star$)
and use Uhlig's cutting procedure to update $\mathcal{G}_j$ and $\mathcal{Z}_j$.
Otherwise, we assume that $q(y)$ does accurately model $\bfovA$ in this region
and do an optimal cut.
To estimate $\theta_\dagger$, we need to determine the line 
\[
	a(y) = a_1 y + a_0
\]
such that $a(y)$ passes through $d_j$ for $y = \Im d_j$ and is tangent to $q(y)$
for some \mbox{$\tilde y \in (0,\Im d_j)$}.
Thus, we solve the following set of equations:
\bseq
\begin{alignat}{4}
	\label{eq:lin1}
	\Re d_j  &= a(\Im d_j)
	 	&& \qquad \Longleftrightarrow \qquad &
		\Re d_j &= a_1 \Im d_j + a_0,\\	
	\label{eq:lin2}
	q(\tilde y) &= a(\tilde y) 
		&& \qquad \Longleftrightarrow \qquad &
		q_2 \tilde y^2 + q_0 &= a_1 \tilde y + a_0, \\
	\label{eq:lin3}
	q^\prime(\tilde y) &= a^\prime(\tilde y) 
		&& \qquad \Longleftrightarrow \qquad &
		2q_2 \tilde y &= a_1,
\end{alignat}
\eseq
to determine $a_0$, $a_1$ and $\tilde y$.  This yields 
\beq
	\label{eq:lincoeffs}
	\tilde y = \Im d_j - \sqrt{(\Im d_j)^2 + \tfrac{q_0 - \Re d_j }{q_2}}, 
	\ \ \
	a_0 = -q_2 \tilde y^2 + q_0, 
	\ \ \  \text{and} \ \ \
	a_1 = \frac{\Re d_j - a_0}{\Im d_j},
\eeq
where $a_1$ follows directly from \eqref{eq:lin1}, $a_0$ is obtained by substituting 
the value of $a_1$ given in \eqref{eq:lin3} into~\eqref{eq:lin2}, and $\tilde y$ follows 
from substituting the value of $a_0$ given in~\eqref{eq:lincoeffs} into $a_1$ in~\eqref{eq:lincoeffs}
(so that $a_1$ now only has $\tilde y$ as an unknown),
and then substituting this version of $a_1$ into~\eqref{eq:lin3}, 
which results in a quadratic equation in~$\tilde y$.
Since $q(y)$ is a sufficiently accurate local model of~$\bfovA$,
it follows that 
\beq
	\label{eq:opt_cut_angle}
	\theta_\dagger \approx \arctan a_1.
\eeq

If $\gamma = \numr$, we can also estimate the value of $\mu$ at~$b_\star$ via
\beq
	\label{eq:mu_est}
	\mu_\mathrm{est} \coloneqq \frac{1}{2|q_2|\gamma},
\eeq
as the osculating circle of $q(y)$ at $y=0$ has radius $\tfrac{1}{2}|q_2|$.
While the value of $\mu$ at $b_\star$ might be computed using \cref{thm:curvature},
this would be much more expensive and it requires that $\lambda_\mathrm{max}(H(0))$ be simple,
which may not hold.
Detecting the normalized radius of curvature at outermost points via~\eqref{eq:mu_est}
will be a key component of our hybrid algorithm in \cref{sec:hybrid}.

Our formulas for computing $\theta_\dagger$ can be
used for any outermost point simply by rotating and flipping the problem
as necessary to satisfy the assumptions on~$b_\star$ and~$b_j$.
To be robust against even small rounding errors, 
instead of $d_j$, we use \mbox{$(1 - \delta) d_j + \delta b_j$} for some small $\delta \in (0,1)$,
\ie a point slightly closer to $b_j$.

For different values of $\mu \in [0,1)$,
\cref{fig:conv_rates} plots the convergence rates for $\{\phi_k\}$ given by \cref{thm:uhlig_angle,thm:opt_conv},
 while \cref{fig:total_cuts} shows the total number of cuts needed by 
each cutting strategy in order to sufficiently approximate~$\bfovX{M}$ near~$b_\star$.
Uhlig's method is usually slightly more expensive
but becomes significantly worse than optimal cutting 
for normalized curvatures $\mu \approx 0.84$ and higher, 
requiring about double the number of cuts at this transition point.
A variant of \cref{fig:total_cuts} (not shown) also reveals that optimal cuts become slightly more expensive
than Uhlig's cuts for~$\mu \approx 0.999961$ and~above, and so we only use the optimal cut
when $\mu_\mathrm{est}$ is less than this value.

\begin{figure}
\centering
\subfloat[Convergence rates of $\{\phi_k\}$.]{
\resizebox*{6.0cm}{!}{\includegraphics[trim=1.3cm 0cm 1.3cm 0cm,clip]{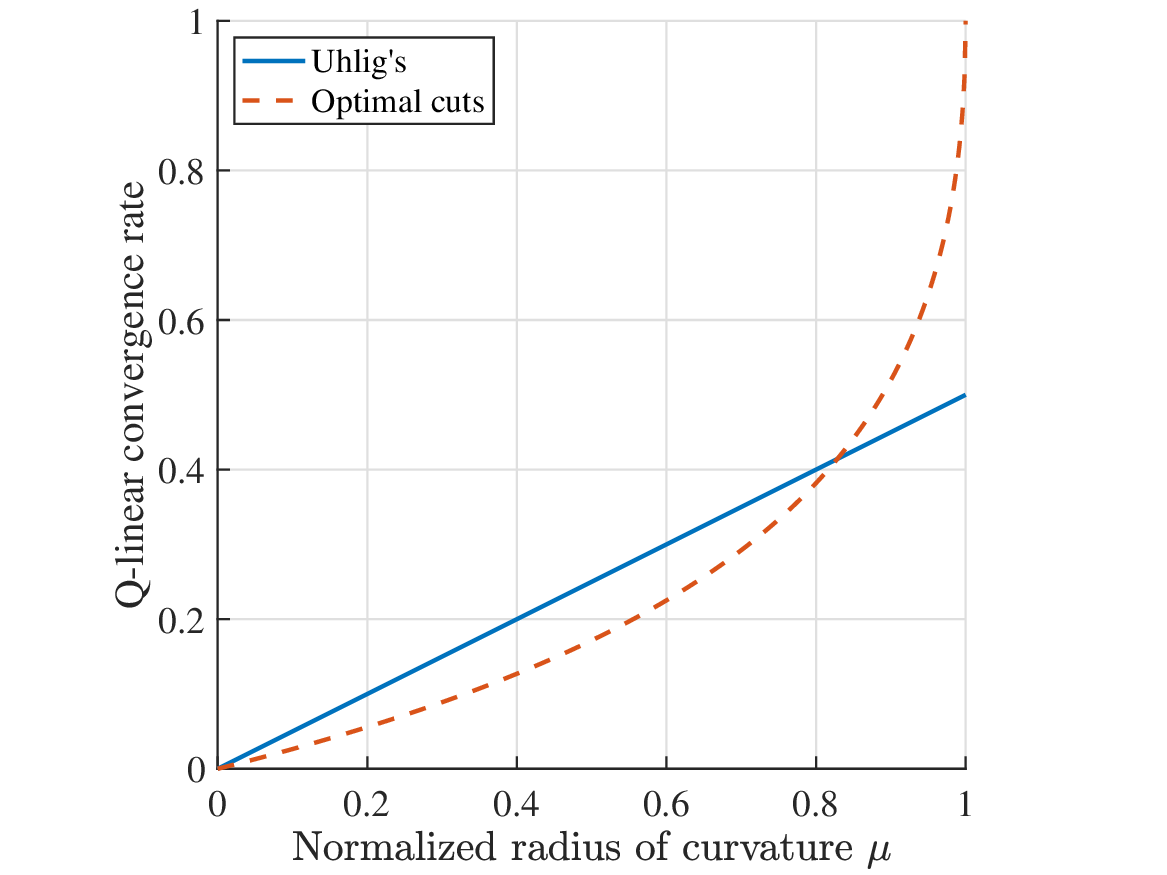}}
\label{fig:conv_rates}
}
\subfloat[Cost to approximate $\bfovX{M}$ region.]{
\resizebox*{6.0cm}{!}{\includegraphics[trim=1.3cm 0cm 1.3cm 0cm,clip]{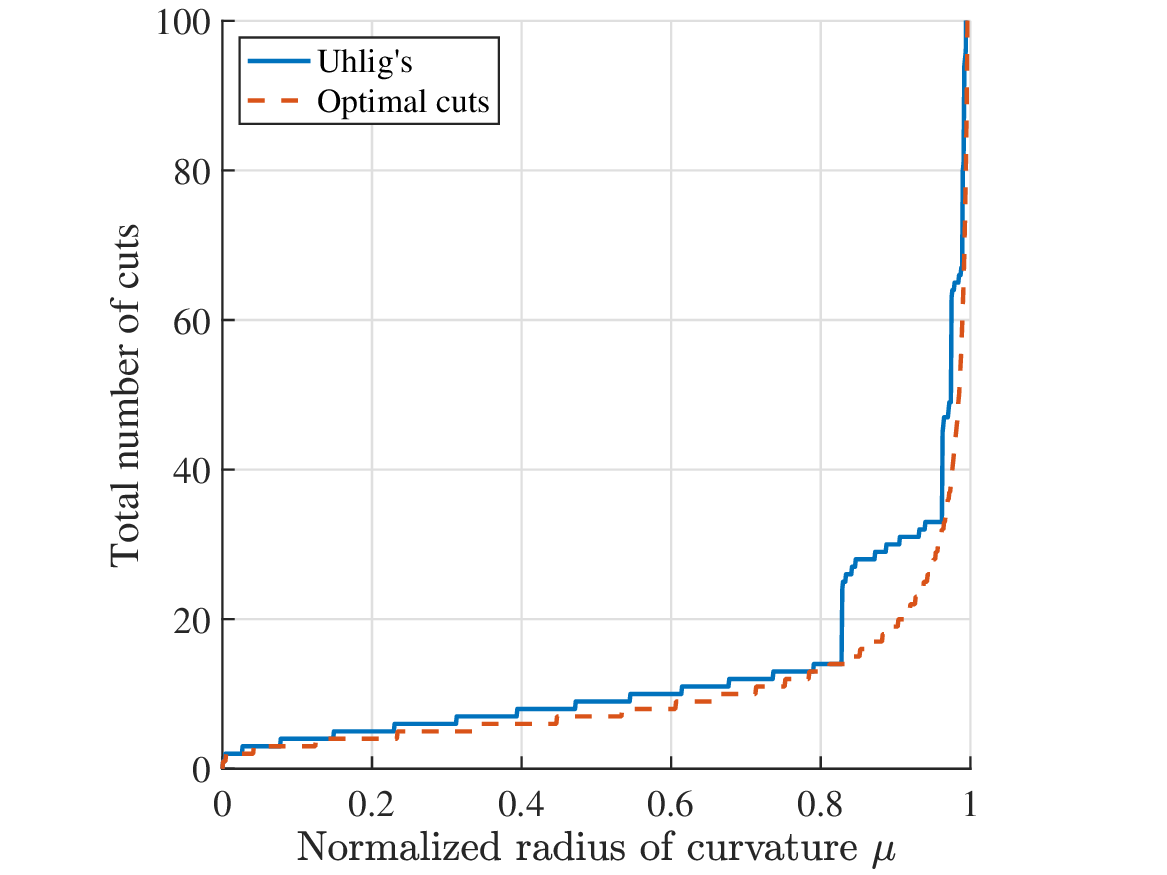}}
\label{fig:total_cuts}
}
\caption{
Left: the respective convergence rates of $\{\phi_j\}$ given by \cref{thm:uhlig_angle,thm:opt_conv}.
Right: for the parametric matrix $M$ with $\mu \in (0,1]$ given by~\cref{ex:circ} and a tolerance \mbox{$\tau_\mathrm{tol} = \texttt{1e-14}$},
the total number of cuts required 
to sufficiently approximate the region of~$\bfovX{M}$ specified by the supporting hyperplanes 
$L_{\theta_\star}$ and~$L_{\theta_j}$ respectively passing through $b_\star$ and~$b_j$
and the corner $c_j$, where $\theta_\star = 0$ and $\theta_j = -\tfrac{\pi}{50}$.
}
 \label{fig:rates_comp}
\end{figure}

\section{A hybrid algorithm}
\label{sec:hybrid}
\cref{tbl:eig} and our new analyses suggest
that it would be far more efficient to combine level-set and cutting-plane techniques in a single hybrid algorithm
rather than rely on either technique alone.
For the smallest values of~$n$, the level-set approach would generally be most efficient,
while for larger problem sizes, which approach would
be fastest depends on the specific shape of $\fovA$ and the normalized 
radius of curvature at outermost points.
While we cannot know these things \emph{a priori}, per \eqref{eq:mu_est},
 \cref{alg2} can estimate $\mu$ as it iterates, and 
 so we can predict how many more cuts may be needed about a particular outermost point.
The current approximation~$\mathcal{G}$ and $\mathcal{Z}$ can also be used to obtain 
cheaply computed estimates of how many more cuts will be needed 
to approximate regions of $\bfovA$ to sufficient accuracy.
Consequently, as \cref{alg2} iterates, 
we can maintain an evolving estimate of how many more cuts would be needed 
in order to compute $\numr$ to the desired tolerance.
Thus, our hybrid 
algorithm can automatically determine if the cutting-plane approach 
is likely to be fast or slow, and if the latter, automatically switch to the level-set approach.
For example, in practice, \cref{alg1} often only requires one to two eigenvalue computations
with $R_\gamma - \lambda S$ and several more with~$H(\theta)$.
Hence, in conjunction with tuning/benchmark data, such as that shown in \cref{tbl:eig},
our hybrid algorithm can reliably estimate whether it will be faster to continue \cref{alg2}
or immediately switch to \cref{alg1}, which will be warm-started using the angle of the supporting hyperplane
that passes through the point in $\mathcal{Z}$ that attains $l$ in line~5 of~\cref{alg2},
as well as the arguments of the corners to the left and right of this point.

\section{Numerical validation}
\label{sec:experiments}
Experiments were done in \matlab\ R2021a (Update 6)
on a 2020 13" MacBook Pro with an Intel i5 1038NG7 quad-core CPU laptop, 16GB of RAM, and macOS v12.4. 
For each code, the desired relative tolerance was set to~$10^{-14}$, and we only report 
errors when they were greater than this amount.
We used \texttt{eig} for all eigenvalue computations\footnote{In practice, note the following recommendations.
As $n$ increases, \texttt{eigs} should be preferred over \texttt{eig} for computing eigenvalues of $H(\theta)$,
but this can be determined automatically via tuning.
Relatedly, we suggest using \texttt{eigs} with $\texttt{k} > 1$ for robustness, 
as the desired eigenvalue may not always be the first to converge.
For robustly identifying all the unimodular eigenvalues of $R_\gamma - \lambda S$, 
it is generally recommended that structure-preserving eigensolvers be used, e.g., \cite{BenBMetal02,BenSV16}.
}   
as (a) it sufficed to verify the benefits of our new methods
and theoretical results and (b) this consistency 
simplifies the comparisons; e.g., \texttt{numr}, Mengi's implementation of his level-set method with Overton, 
also only uses \texttt{eig}.  
All code and data are included as supplementary material for reproducibility.
Implementations of our new methods will also be added to ROSTAPACK~\cite{rostapack}.

We begin by comparing \cref{alg1} to \texttt{numr}.
Per \cref{tbl:level},
\cref{alg1} generally only needed a single eigenvalue computation with $R_\gamma - \lambda S$ and at most two,
and for $n \geq 300$, ranged from 4.3--6.9 times faster than \texttt{numr}.
Even with optimization disabled, our iteration using $\rho(H(\theta))$ 
was still faster than~\texttt{numr}.

\begin{table}
\centering
\footnotesize
\caption{
For dense random $A$ matrices,
the costs of \cref{alg1}
and Mengi's code \texttt{numr} (MO, for Mengi and Overton's method) are shown.
We tested \cref{alg1} with optimization enabled (Opt., done via Newton's method) and disabled (Mid., for midpoints only).
}
\setlength{\tabcolsep}{6pt} 
\begin{tabular}{r SSc | ccc | SSS} 
\toprule
\multicolumn{1}{c}{} &
\multicolumn{3}{c}{\# of \texttt{eig($R_\gamma$,$S$)}} & 
\multicolumn{3}{c}{\# of \texttt{eig($H(\theta)$)}} &
\multicolumn{3}{c}{Time (sec.)}  \\
\cmidrule(lr){2-4}
\cmidrule(lr){5-7}
\cmidrule(lr){8-10}
\multicolumn{1}{c}{} & 
\multicolumn{2}{c}{Alg.~\ref{alg1}} & 
\multicolumn{1}{c}{MO} & 
\multicolumn{2}{c}{Alg.~\ref{alg1}} & 
\multicolumn{1}{c}{MO} & 
\multicolumn{2}{c}{Alg.~\ref{alg1}} & 
\multicolumn{1}{c}{MO} \\ 
\cmidrule(lr){2-3}
\cmidrule(lr){4-4}
\cmidrule(lr){5-6}
\cmidrule(lr){7-7}
\cmidrule(lr){8-9}
\cmidrule(lr){10-10}
\multicolumn{1}{c}{$n$} & 
\multicolumn{1}{c}{Opt.} & 
\multicolumn{1}{c}{Mid.} & 
\multicolumn{1}{c}{} & 
\multicolumn{1}{c}{Opt.} & 
\multicolumn{1}{c}{Mid.} & 
\multicolumn{1}{c}{} & 
\multicolumn{1}{c}{Opt.} & 
\multicolumn{1}{c}{Mid.} & 
\multicolumn{1}{c}{} \\ 
\midrule
\multicolumn{1}{r|}{ 100} &  1 &  4 &  6 &  9 & 10 & 39 &   0.1 &   0.2 &   0.2 \\ 
\multicolumn{1}{r|}{ 200} &  2 &  4 &  5 & 14 & 11 & 34 &   0.6 &   1.0 &   1.2 \\ 
\multicolumn{1}{r|}{ 300} &  1 &  4 &  5 & 17 &  9 & 26 &   1.2 &   3.9 &   4.9 \\ 
\multicolumn{1}{r|}{ 400} &  1 &  5 &  7 & 17 & 15 & 42 &   2.7 &  11.7 &  16.0 \\ 
\multicolumn{1}{r|}{ 500} &  1 &  3 &  7 &  7 &  8 & 54 &   4.5 &  12.8 &  30.7 \\ 
\multicolumn{1}{r|}{ 600} &  1 &  4 &  7 &  8 & 10 & 45 &   7.7 &  28.7 &  51.4 \\ 
\bottomrule
\end{tabular} 
\label{tbl:level}
\end{table}

We now verify that our local convergence rate analyses from \cref{sec:uhlig_conv} and \cref{sec:opt_conv}
do indeed hold for general matrices and that our procedure for computing optimal cuts is sufficiently accurate
to realize the convergence rate given by \cref{thm:opt_conv}.
First, we obtained 200 general examples with roughly equally spaced values of~\mbox{$\mu \in [0,1]$}.
This was done by running optimization on $\min_{X} r(A+BXC)$, 
where $A \in \C^{10 \times 10}$ is diagonal, while $B \in \C^{10 \times 5}$, $C \in \C^{5 \times 10}$,
and $X \in \R^{5 \times 5}$ are dense, and collecting the iterates.
By starting at $X=0$, we obtain an example with~$\mu=0$; since~$A$ is diagonal, $\fovA$ is a polygon. 
Since minimizing $r(A+BXC)$  
often causes $\mu \to 1$ as optimization progresses~\cite{LewO20},
we also obtain a sequence of examples $A+BX_kC$ for iterates $\{X_k\}$ with various~$\mu$~values
(computed via \cref{thm:curvature}).
Generating new $A$, $B$, and $C$ matrices and running optimization from $X_0=0$ was repeated 
in a loop until the desired set of 200 general examples had been obtained.
For each problem, we recorded the total number of cuts
that Uhlig's cutting procedure and the optimal-cut strategy needed to 
sufficiently approximate the field of values boundary in a small neighborhood to one side of the outermost point in its field of values.  More specifically, 
we performed an analogous experiment to the one we showed earlier for approximating 
a region of the boundary of \cref{ex:circ}.
As can be seen by comparing \cref{fig:total_cuts,fig:alg2_exp}, for any given $\mu$, 
the total number of cuts needed on arbitrarily shaped fields of values 
is essentially the same as that needed for \cref{ex:circ},
thus validating the generality of our convergence rate analysis and the reliability of our method for computing optimal cuts.

\begin{figure}[t]
\centering
\resizebox*{6.0cm}{!}{\includegraphics[trim=1.3cm 0cm 1.3cm 0cm,clip]{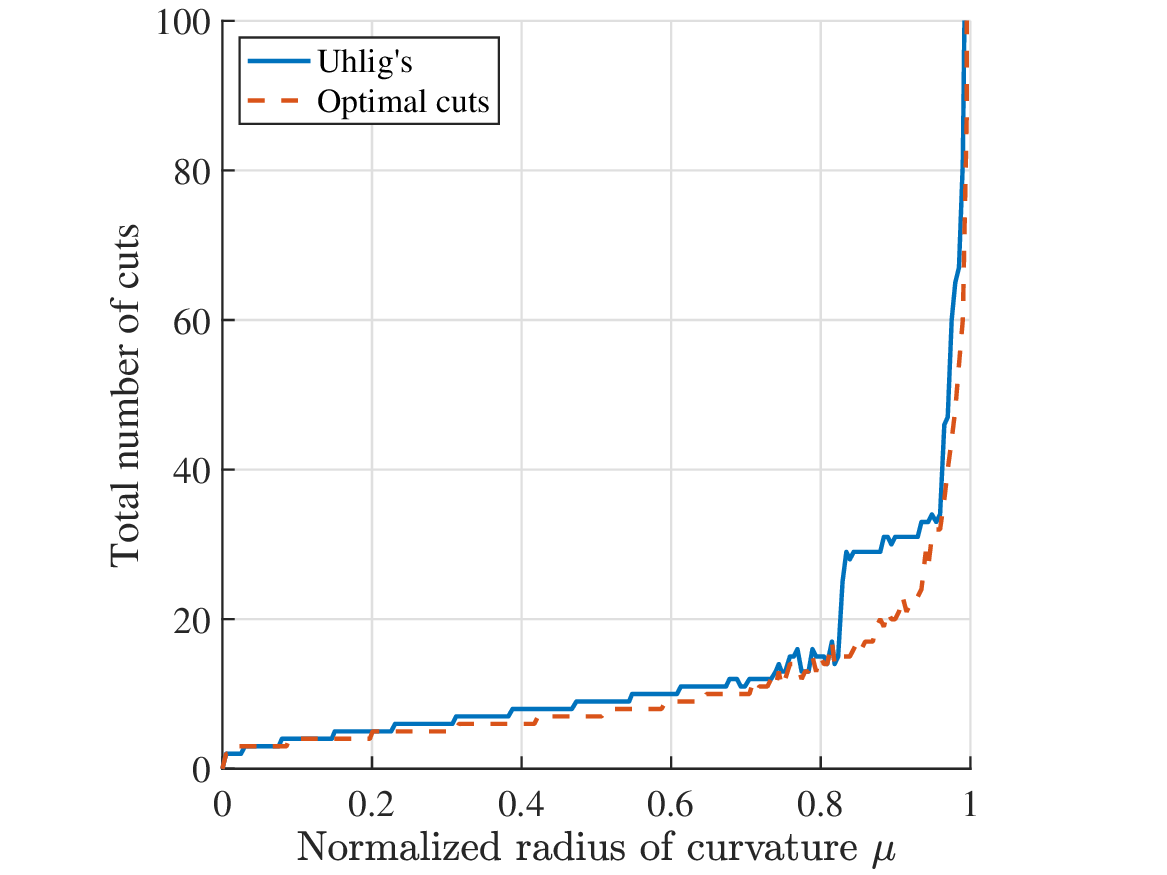}}
\caption{
Using 200 different general matrices of the form $A+BXC$ with arbitrarily shaped fields of values and $\mu$ values roughly equally spaced between $0$ and $1$ at the outermost points in $\fovX{A+BXC}$,
Uhlig's cutting procedure and optimal cuts are compared.
For each example with a different value of~$\mu$, 
we plot the total number of cuts needed to approximate the region of~$\bd\fovX{A+BXC}$ specified by 
supporting hyperplanes~$L_{\theta_\star}$ and~$L_{\theta_j}$ and the corner $c_j$ of $\mathcal{G}_j$ that they define, 
where~$L_{\theta_\star}$ passes through the outermost point~$b_\star$ and~\mbox{$\theta_j = \theta_\star - \tfrac{\pi}{50}$}.
The strong agreement of this plot with that of~\cref{fig:total_cuts}
empirically validates the generality of our convergence rate analyses,
as~\cref{fig:total_cuts} gives the analogous experiments for approximating 
the ``same-sized" region of~$\bd\fovX{M}$ (in the sense that~\mbox{$\theta_\star - \theta_j = \tfrac{\pi}{50}$})
for the parametric matrix~$M$ given by~\cref{ex:circ} with $\mu \in (0,1]$.
}
 \label{fig:alg2_exp}
\end{figure}

For comparing our improved level-set and cutting-plane methods,
 we also set \cref{alg2} to do optimization via Newton's method, and per \cref{rem:two_hp},
had it add supporting hyperplanes for both $\lambda_\mathrm{max}$ 
and $\lambda_\mathrm{min}$ on every cut.
For test problems, we used
the Gear, Grcar, and FM examples used by Uhlig in \cite{Uhl09}, 
 \texttt{randn}-based complex matrices,
and  $\e^{\imagunit 0.25\pi}((1 - \mu) I + \mu K_n)$ with
\mbox{$\mu=0.9999$} and $K_n$ from~\eqref{eq:crabb},
which is a rotated version of \cref{ex:circ} that we call Nearly Disk.
In~\cref{tbl:both}, we again see that \cref{alg1} 
is well optimized in terms of its overall possible efficiency, 
as it often only required a single computation with $R_\gamma - \lambda S$ 
and at most two.  As predicted by our analysis,
we also see that the cost of \cref{alg2} is highly correlated
with the value of $\mu$. 
On Gear ($\mu \approx 0$), \cref{alg2} was extremely fast, essentially showing Q-superlinear convergence.
In fact, on the Gear, Grcar, FM, and \texttt{randn} matrices, 
\cref{alg2} was much faster (3.4 to 234.2 times) than \cref{alg1} as $\mu < 0.9$ for all of these problems.
In contrast, for Nearly Disk ($\mu=0.9999$),
 \cref{alg2} was noticeably slower, with our level-set approach now being 5.7 to~11.5 times faster.

\begin{table}
\centering
\footnotesize 
\caption{
The  respective costs of \cref{alg1,alg2} are shown.
The values of $\mu$ at outermost points are also shown,
computed via \cref{thm:curvature}. 
}
\setlength{\tabcolsep}{6pt} 
\begin{tabular}{lrc | cS | r | SS } 
\toprule
\multicolumn{1}{c}{} & 
\multicolumn{1}{c}{} & 
\multicolumn{1}{c}{} & 
\multicolumn{3}{c}{\# of calls to \texttt{eig($\cdot$)}} &
\multicolumn{2}{c}{Time (sec.)}  \\
\cmidrule(lr){4-6}
\cmidrule(lr){7-8}
\multicolumn{1}{c}{} & 
\multicolumn{1}{c}{} &
\multicolumn{1}{c}{} & 
\multicolumn{2}{c}{Alg.~\ref{alg1}} & 
\multicolumn{1}{c}{Alg.~\ref{alg2}} & 
\multicolumn{1}{c}{Alg.~\ref{alg1}} & 
\multicolumn{1}{c}{Alg.~\ref{alg2}} \\
\cmidrule(lr){4-5}
\cmidrule(lr){6-6}
\cmidrule(lr){7-7}
\cmidrule(lr){8-8}
\multicolumn{1}{l}{Problem} & 
\multicolumn{1}{c}{$n$} & 
\multicolumn{1}{c}{$\mu$} & 
\multicolumn{1}{c}{$R_\gamma - \lambda S$} & 
\multicolumn{1}{c}{$H(\theta)$} & 
\multicolumn{1}{c}{$H(\theta)$} & 
\multicolumn{1}{c}{} & 
\multicolumn{1}{c}{} \\ 
\midrule
Gear            & \multicolumn{1}{|r|}{ 320} & $1.194 \times 10^{-6}$  &   1  &   2  &    4  &   1.2  &   0.1 \\ 
Gear            & \multicolumn{1}{|r|}{ 640} & $1.499 \times 10^{-7}$  &   1  &   5  &    3  &  15.2  &   0.1 \\ 
Gear            & \multicolumn{1}{|r|}{1280} & $1.878 \times 10^{-8}$  &   1  &   5  &    3  & 224.6  &   1.0 \\ 
\midrule
Grcar           & \multicolumn{1}{|r|}{ 320} & $              0.6543$  &   1  &  32  &   28  &   1.5  &   0.3 \\ 
Grcar           & \multicolumn{1}{|r|}{ 640} & $              0.6544$  &   1  &  29  &   29  &  15.3  &   1.7 \\ 
Grcar           & \multicolumn{1}{|r|}{1280} & $              0.6544$  &   1  &  27  &   29  & 215.0  &  12.8 \\ 
\midrule
FM              & \multicolumn{1}{|r|}{ 320} & $              0.1851$  &   1  &  11  &   19  &   1.3  &   0.1 \\ 
FM              & \multicolumn{1}{|r|}{ 640} & $              0.1836$  &   1  &   8  &   18  &  10.9  &   0.5 \\ 
FM              & \multicolumn{1}{|r|}{1280} & $              0.1829$  &   1  &   9  &   18  &  87.7  &   3.3 \\ 
\midrule
\texttt{randn}  & \multicolumn{1}{|r|}{ 320} & $              0.7576$  &   1  &  15  &   40  &   1.6  &   0.5 \\ 
\texttt{randn}  & \multicolumn{1}{|r|}{ 640} & $              0.8663$  &   2  &  17  &   67  &  24.7  &   3.7 \\ 
\texttt{randn}  & \multicolumn{1}{|r|}{1280} & $              0.7971$  &   2  &  23  &   50  & 203.5  &  22.0 \\ 
\midrule
Nearly Disk     & \multicolumn{1}{|r|}{ 320} & $              0.9999$  &   1  &   6  & 1570  &   1.7  &  20.0 \\ 
Nearly Disk     & \multicolumn{1}{|r|}{ 640} & $              0.9999$  &   1  &   6  & 1561  &  13.2  &  88.6 \\ 
Nearly Disk     & \multicolumn{1}{|r|}{1280} & $              0.9999$  &   1  &   7  & 1556  & 114.4  & 647.2 \\ 
\bottomrule
\end{tabular} 
\label{tbl:both}
\end{table}

Finally, we benchmark our hybrid algorithm and begin by  comparing
it with~\cref{alg1,alg2}.
We tested our three algorithms on $n=400$ and $n=800$ examples 
with values 
$\mu \in \{0.1,0.2,\ldots,0.9,0.99,\ldots,0.9999999\}$ in order
to represent the range of normalized radius of curvatures 
that may be encountered when minimizing the numerical radius.
Since minimizing~$r(A+BXC)$ to generate such matrices would be prohibitively expensive
for these values of $n$ and $\mu$, 
we instead generated examples of the 
form~\mbox{$T_{n,\mu} = \eit \begin{bsmallmatrix} M & 0 \\ 0 & D\end{bsmallmatrix}$}, where 
$\theta \in [0,2\pi)$ was chosen randomly, $M$ is an instance of \cref{ex:circ}
 with the desired value of $\mu$ and dimension $n-100$, and 
$D \in \C^{100 \times 100}$ is a complex diagonal matrix.
In order to make $h(\theta)$ have many local maximizers, 
we chose $M$ such that~$r(M)=1$ and then picked the elements of $D$ so that
they were roughly placed near a circle drawn between $\bfovX{\eit M}$ and the unit circle,
biased towards the latter;
see \cref{fig:hybrid_ex_fov} for a visualization.
\cref{fig:hybrid_ex_htheta} shows how this choice of~$D$ 
indeed causes~$h(\theta)$ to have many local maximizers, 
while the randomly chosen~$\eit$ scalar
means that the unique global maximizer may occur anywhere.
The running times of our three algorithms on these $T_{n,\mu}$ examples are
shown in \cref{fig:hybrid_comp}.
Once again, we see that the running time of \cref{alg1} remains fairly constant
across all the values of $\mu$, while the running time of \cref{alg2} is much faster 
for small values of $\mu$ but then blows up as $\mu \to 1$.  
Most importantly, \cref{fig:hybrid_comp} verifies that our hybrid algorithm indeed remains efficient 
for all values of~$\mu$ since it automatically detects when to switch from 
the cutting-plane approach to the level-set approach.
In fact, our hybrid algorithm even becomes more efficient than \cref{alg1} for $\mu$ close to one.
This is because when it switches to the level-set approach, \cref{alg2} often provided 
such good starting points
that only one eigenvalue computation with~$R_\gamma - \lambda S$ was needed.
In contrast, \cref{alg1} always required two eigenvalue computations  with $R_\gamma - \lambda S$ 
on our $T_{n,\mu}$ test problems.

\begin{figure}
\centering
\subfloat[$n=400$, $\mu = 0.4$.]{
\resizebox*{6.0cm}{!}{\includegraphics[trim=1.3cm 0cm 1.3cm 0cm,clip]{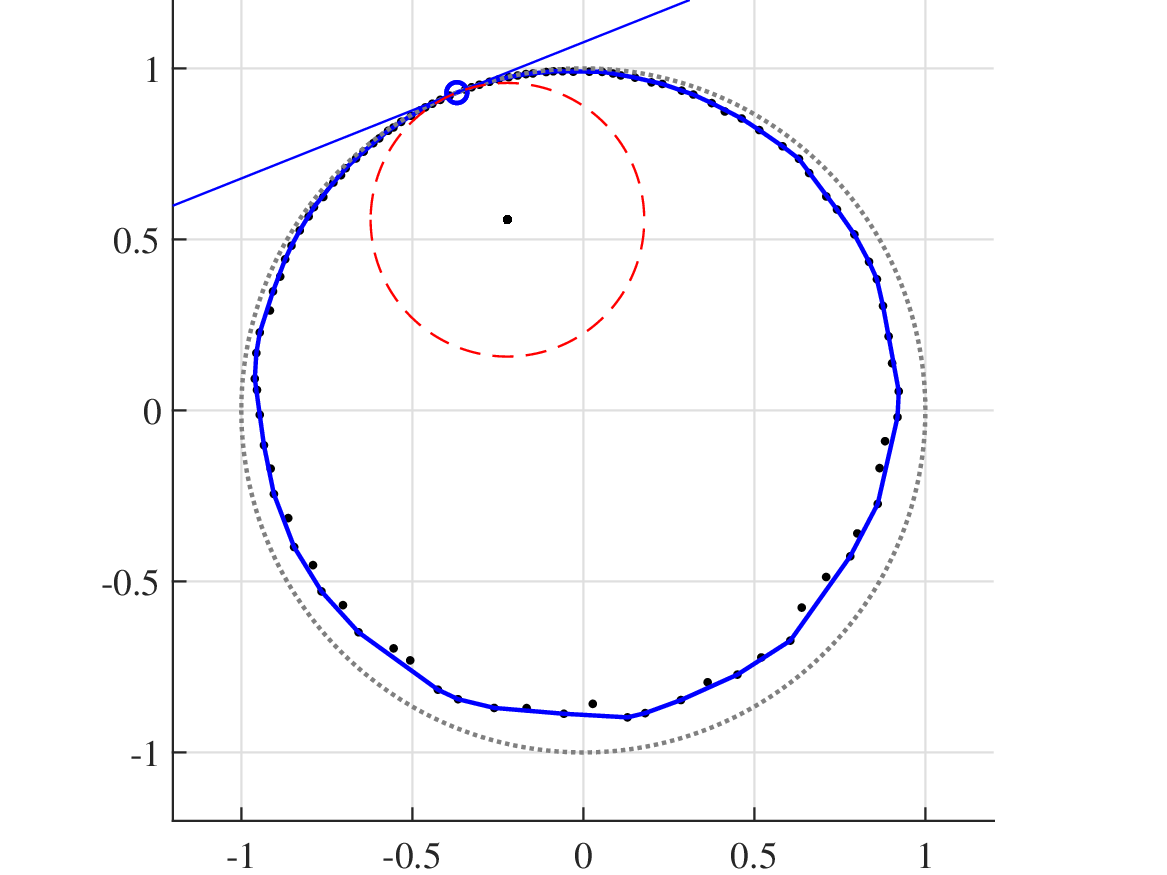}}
\label{fig:hybrid_ex_fov}
}
\subfloat[$n=400$, $\mu = 0.99$]{
\resizebox*{6.0cm}{!}{\includegraphics[trim=1.3cm 0cm 1.3cm 0cm,clip]{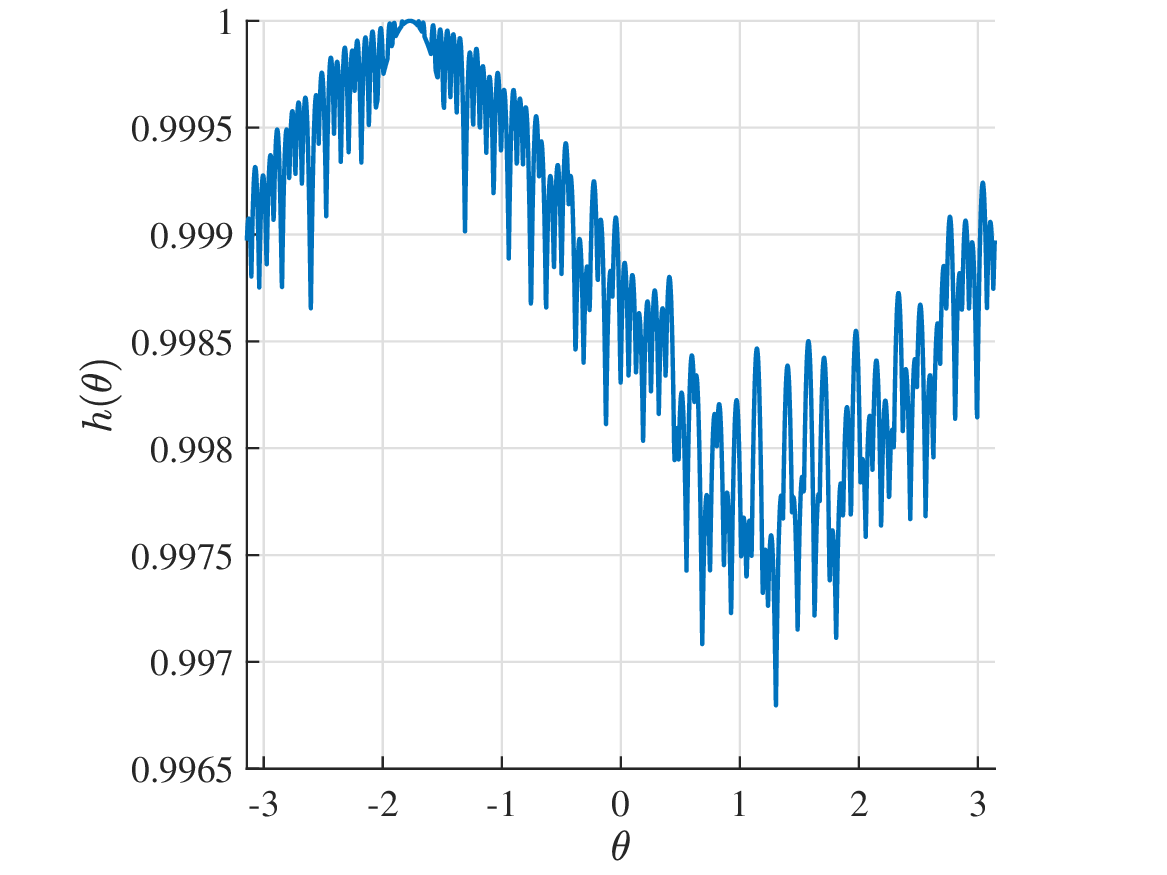}}
\label{fig:hybrid_ex_htheta}
}
\caption{
Left: the field of values, eigenvalues, and the osculating circle at the unique point attaining the
numerical radius are shown for one instance of $T_{n,\mu}$; 
to read the plot, see the caption of \cref{fig:demo}.
Right: a plot of $h(\theta)$ for another instance of $T_{n,\mu}$.
}
 \label{fig:hybrid_exs}
\end{figure}

\begin{figure}
\centering
\subfloat[$n=400$.]{
\resizebox*{6.16cm}{!}{\includegraphics[trim=0.06cm 0cm 0.25cm 0cm,clip]{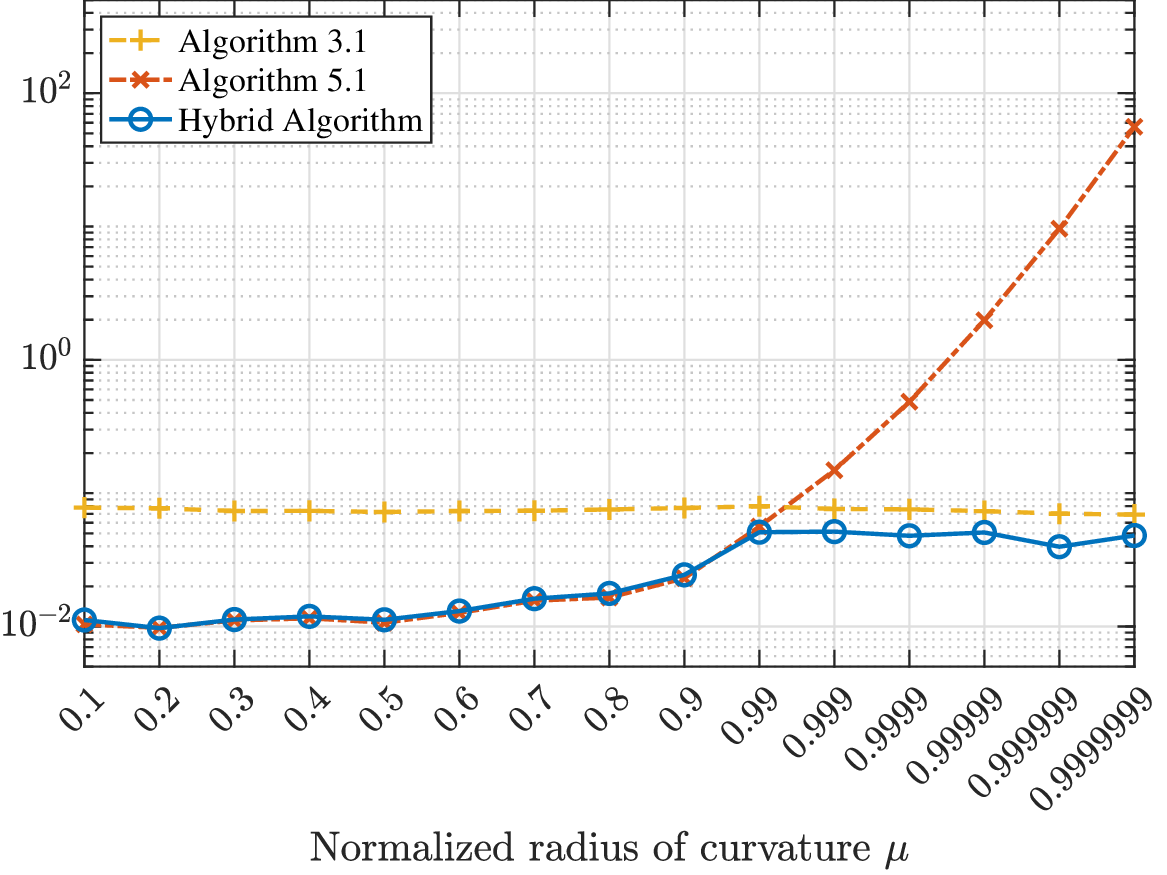}}
\label{fig:hybrid_400}
}
\subfloat[$n=800$]{
\resizebox*{6.16cm}{!}{\includegraphics[trim=0.06cm 0cm 0.25cm 0cm,clip]{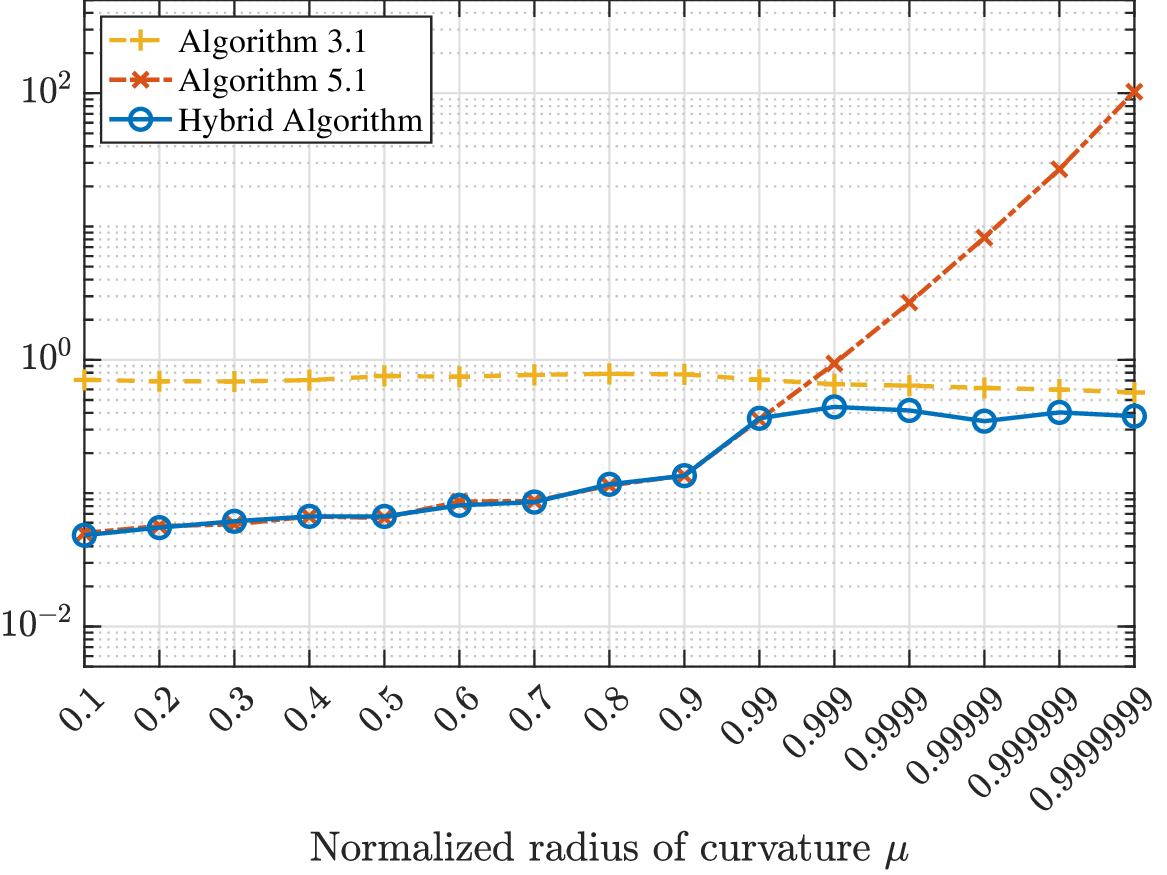}}
\label{fig:hybrid_800.}
}
\caption{
Running times (in minutes, $\log_{10}$ scale) of the algorithms on  $T_{n,\mu}$ with different normalized curvatures.
}
 \label{fig:hybrid_comp}
\end{figure}

For our last set of experiments, we compare our hybrid algorithm
with Mengi's \texttt{numr} code and Uhlig's \texttt{NumRadius} routine,
which respectively implement the older methods described in~\cite{MenO05} and \cite{Uhl09}.
As our hybrid algorithm allows up to~$10000$ cutting-plane iterations by default,
we also set the maximum allowed number of iterations of \texttt{NumRadius} to $10000$ (its default is only $200$).
Furthermore, we set \texttt{NumRadius} to use \texttt{eig}, since we have standardized on using \texttt{eig} for all experiments.
We tested the three codes on 39 matrices, all $800 \times 800$ in size, with 20 taken from EigTool~\cite{eigtool},
and another 19 taken from \texttt{gallery} in \matlab;
the matrices have a variety of different $\mu$ values at outermost points in the field of values,
ranging from~$\mu$ close to~$0$ to~$\mu = 1$.
\cref{tbl:codecomp} reports the performance data of the three codes on these problems, and from the table, 
it is immediately apparent that our hybrid algorithm remains efficient across all problems, while \texttt{numr} is quite slow on most of the problems,
and \texttt{NumRadius} becomes more and more inefficient as~$\mu \to 1$, as expected.
Often our hybrid method is also the fastest method on each problem, and even when it is not,
its performance is generally quite close to the fastest code, precisely because our algorithm
automatically adapts to each problem.
For problems with~$\mu < 0.831$, our hybrid is generally neck and neck with \texttt{NumRadius}, while being significantly
faster than \texttt{numr}.
Meanwhile, for problems with~$\mu \geq 0.831$, i.e., problems for which optimal cuts and our hybrid algorithm have the most benefit,
our hybrid algorithm is generally notably faster than \texttt{NumRadius},
and becomes dramatically faster for the three problems with~$\mu = 1.000$.
Our hybrid algorithm is also much faster than \texttt{numr} on these~$\mu \geq 0.831$ problems,
except for the three $\mu = 1.000$ examples, where \texttt{numr} ranged from 1.1--1.9 times faster.  
In some cases, we see that our method reduces the total number of eigenvalues computations with $H(\theta)$ (with respect to the total done by \texttt{NumRadius})
more than what is suggested by comparing the respective running times, e.g., on \texttt{'randcolu'};
we believe this is because our \matlab\ implementation is significantly more complex than \texttt{NumRadius} and so has a higher overhead.
Overall, we see that the experiments in~\cref{tbl:codecomp} paint a similar picture to our earlier experiments shown in~\cref{fig:hybrid_comp}.
Using a compiled language and parallelism may help alleviate these remaining issues.
Finally, we briefly discuss the accuracy of the computed estimates.  On most of the problems, all three codes produced
estimates with at least 14 digits of agreement.  However, for the three $\mu=1.000$ examples, \texttt{NumRadius}
reached the maximum of $10000$ iterations, and so its upper bound could only certify that seven or eight digits were correct.
Meanwhile, when \texttt{numr} only performed one level-set test before terminating, we noticed that it returned estimates that were slightly too high,
with errors in the tenth most significant digit; this appears to be a minor bug in the \texttt{numr} routine.

\begin{table}
\centering
\footnotesize 
\caption{
The performance data of Mengi's \texttt{numr} (MO), Uhlig's \texttt{NumRadius} (U), and
our hybrid algorithm (Hy.) are shown for various $800 \times 800$ matrices from EigTool and \texttt{gallery} in \matlab\ 
(matrices from \texttt{gallery} are shown in single-quotes in table).  
The problems are sorted in order of increasing $\mu$ value 
(all $\mu$ values are positive and shown to the first three decimal places),
and the table is divided at $\mu \approx 0.831$, above which optimal cuts
and our hybrid algorithm have the most benefits.
The fastest running time per problem is typeset in bold.
}
\setlength{\tabcolsep}{4.5pt} 
\begin{tabular}{l | c | cc | rrr | rrr } 
\toprule
\multicolumn{2}{c}{} &
\multicolumn{5}{c}{\# of calls to \texttt{eig($\cdot$)}} &
\multicolumn{3}{c}{} \\
\cmidrule(lr){3-7}
\multicolumn{1}{c}{} &
\multicolumn{1}{c}{} &
\multicolumn{2}{c}{$R_\gamma - S$} & 
\multicolumn{3}{c}{$H(\theta)$} &
\multicolumn{3}{c}{Time (sec.)} \\
\cmidrule(lr){3-4}
\cmidrule(lr){5-7}
\cmidrule(lr){8-10}
\multicolumn{1}{c}{Problem} & 
\multicolumn{1}{c}{$\mu$} & 
\multicolumn{1}{c}{MO} & 
\multicolumn{1}{c}{Hy.} & 
\multicolumn{1}{c}{MO} & 
\multicolumn{1}{c}{U} & 
\multicolumn{1}{c}{Hy.} & 
\multicolumn{1}{c}{MO} & 
\multicolumn{1}{c}{U} & 
\multicolumn{1}{c}{Hy.} \\
\midrule
gausseidel(C)     & 0.000 &    1 &    0 &      1 &      3 &      3 &            23.1 &             0.4 &    \textbf{0.4} \\ 
\texttt{'chebvand'}      & 0.000 &    6 &    0 &     97 &     20 &     15 &            82.8 &             1.8 &    \textbf{1.5} \\ 
\texttt{'dorr'}          & 0.000 &    1 &    0 &      9 &      4 &      4 &            19.3 &    \textbf{0.3} &             0.4 \\ 
twisted                  & 0.002 &    3 &    0 &    692 &     12 &      7 &           102.3 &             1.2 &    \textbf{0.8} \\ 
basor                    & 0.007 &    3 &    0 &     14 &      9 &     12 &            57.6 &    \textbf{1.0} &             1.3 \\ 
convdiff                 & 0.022 &    2 &    0 &   3256 &      6 &      5 &           224.2 &             0.5 &    \textbf{0.5} \\ 
davies                   & 0.022 &    1 &    0 &   1617 &      9 &     11 &           112.5 &    \textbf{0.9} &             1.1 \\ 
airy                     & 0.022 &    2 &    0 &   3209 &      8 &     11 &           224.9 &    \textbf{0.8} &             1.1 \\ 
landau                   & 0.047 &    3 &    0 &     14 &     30 &     31 &            46.4 &    \textbf{2.7} &             2.9 \\ 
hatano                   & 0.085 &    1 &    0 &      1 &      7 &      7 &            15.4 &    \textbf{0.6} &             0.7 \\ 
\texttt{'clement'}       & 0.131 &    1 &    0 &      9 &     26 &     13 &            19.0 &             2.4 &    \textbf{1.4} \\ 
\texttt{'redheff'}       & 0.155 &    1 &    0 &      5 &      8 &      8 &            14.2 &    \textbf{0.7} &             0.8 \\ 
\texttt{'riemann'}       & 0.284 &    1 &    0 &      5 &     10 &      9 &            17.9 &             1.0 &    \textbf{0.9} \\ 
\texttt{'lesp'}          & 0.330 &    2 &    0 &   1608 &     11 &     11 &           140.0 &    \textbf{1.0} &             1.1 \\ 
gausseidel(D)    & 0.392 &    1 &    0 &      1 &     12 &     10 &            15.6 &             1.2 &    \textbf{1.0} \\ 
\texttt{'jordbloc'}      & 0.500 &    1 &    0 &      1 &     13 &     12 &            21.3 &    \textbf{1.2} &             1.2 \\ 
transient                & 0.615 &    2 &    0 &   1117 &     28 &     28 &           114.3 &             2.7 &    \textbf{2.7} \\ 
frank                    & 0.641 &    1 &    0 &      5 &     16 &     14 &            12.4 &             1.6 &    \textbf{1.4} \\ 
grcar                    & 0.654 &    7 &    0 &    333 &     34 &     28 &           159.5 &             3.4 &    \textbf{2.9} \\ 
\texttt{'dramadah'}      & 0.659 &    1 &    0 &      5 &     16 &     14 &            13.9 &             1.5 &    \textbf{1.5} \\ 
\texttt{'chow'}          & 0.664 &    1 &    0 &      5 &     16 &     15 &            13.9 &    \textbf{1.5} &             1.6 \\ 
\texttt{'triw'}          & 0.669 &    2 &    0 &   1606 &     17 &     16 &           128.8 &    \textbf{1.6} &             1.6 \\ 
chebspec                 & 0.727 &    1 &    0 &      5 &     18 &     17 &            13.1 &    \textbf{1.7} &             1.7 \\ 
kahan                    & 0.741 &    2 &    0 &     11 &     18 &     18 &            28.9 &    \textbf{1.7} &             1.9 \\ 
\texttt{'cycol'}         & 0.807 &    7 &    0 &     69 &     54 &     59 &            88.6 &    \textbf{5.0} &             6.0 \\ 
gausseidel(U)      & 0.818 &    1 &    0 &      1 &     22 &     21 &            23.0 &             5.9 &    \textbf{5.8} \\ 
\midrule
riffle                   & 0.831 &    1 &    0 &      1 &     36 &     22 &            13.9 &             3.7 &    \textbf{2.2} \\ 
\texttt{'randcolu'}      & 0.837 &    5 &    0 &     63 &     62 &     53 &            78.2 &             5.8 &    \textbf{5.6} \\ 
random                   & 0.843 &    5 &    0 &     45 &     72 &     56 &            77.9 &             7.1 &    \textbf{6.0} \\ 
\texttt{'lotkin'}        & 0.887 &    1 &    0 &      1 &     41 &     28 &            11.8 &             4.0 &    \textbf{2.8} \\ 
\texttt{'randjorth'}     & 0.924 &    5 &    0 &     55 &     82 &     62 &            64.4 &             8.0 &    \textbf{6.5} \\ 
\texttt{'leslie'}        & 0.929 &    1 &    0 &      1 &     43 &     33 &            13.5 &             4.2 &    \textbf{3.5} \\ 
\texttt{'randsvd'}       & 0.934 &    2 &    0 &     11 &     49 &     45 &            26.3 &             4.6 &    \textbf{4.4} \\ 
orrsommerfeld            & 0.935 &    3 &    0 &   4024 &     85 &     77 &           297.2 &             8.2 &    \textbf{7.5} \\ 
randomtri                & 0.944 &    6 &    0 &     54 &     97 &     75 &            84.3 &            10.0 &    \textbf{7.8} \\ 
demmel                   & 0.998 &    2 &    0 &     11 &    258 &    233 &            25.9 &            24.9 &   \textbf{24.2} \\ 
\texttt{'forsythe'}      & 1.000 &    1 &    1 &      1 &  10000 &     35 &   \textbf{13.6} &           957.2 &            26.1 \\ 
\texttt{'smoke'}         & 1.000 &    1 &    1 &      1 &  10000 &     20 &   \textbf{18.9} &           960.1 &            21.1 \\ 
\texttt{'parter'}        & 1.000 &    1 &    1 &      1 &  10000 &     61 &   \textbf{20.8} &           971.7 &            26.8 \\ 
\midrule 
\multicolumn{7}{r}{Total Time:} & \multicolumn{1}{r}{2479.8} & \multicolumn{1}{r}{3013.8} & \multicolumn{1}{r}{188.8} \\ 
\bottomrule
\end{tabular} 
\label{tbl:codecomp}
\end{table}

\section{Conclusion}
\label{sec:conclusion}
Via our new understanding of the local convergence rate of Uhlig's cutting procedure, as well as how
the overall cost of his method blows up for disk matrices, we have 
precisely explained why Uhlig's method is sometimes much faster or much slower
than the level-set method of Mengi and Overton.  Moreover, this analysis has motivated 
our new hybrid algorithm that automatically switches between cutting-plane and level-set techniques
in order to remain efficient across all numerical radius problems.
Along the way, we have also identified inefficiencies in the earlier level-set and cutting-plane algorithms
and addressed them via our improved versions of these two methodologies.

\clearpage

\bibliographystyle{alpha} 
\bibliography{csc,software}
\end{document}